\documentclass[review,onefignum,onetabnum]{siamart171218}

\usepackage{physics}
\usepackage{mathtools}
\usepackage{subfig}
\usepackage{tikz-cd}
\usepackage{tikz}
\usepackage{amssymb}
\usepackage{mathrsfs}
\usepackage{pgfplots}
\pgfplotsset{compat=1.17}

\newtheorem{ass}{Assumption}
\newtheorem{thm}{Theorem} 

\newtheorem{rmk}{Remark}

\newtheorem{lem}{Lemma}

\newcommand\eqdef{\mathrel{\stackrel{\makebox[0pt]{\mbox{\normalfont\tiny def}}}{=}}}
\NewDocumentCommand{\llm}{ O{2} O{\Xi}}{\widehat{\mathfrak{L}\pi}_{#1,#2}}
\NewDocumentCommand{\llmper}{ O{2} O{\Xi}}{\widehat{\mathfrak{L}\pi}_{#1,#2}^{per}}
\NewDocumentCommand{\llmperc}{ O{2} O{\Xi}}{\widehat{\mathfrak{L}\pi}_{#1,#2}^{c,per}}
\NewDocumentCommand{\LLM}{ O{i} }{\widehat{\mathfrak{L}\Pi}^{#1}}
\NewDocumentCommand{\pLLM}{ O{i} }{{\mathfrak{L}\Pi}^{#1}}
\NewDocumentCommand{\LLMper}{ O{i} }{{\widehat{\mathfrak{L}\Pi}^{#1}}}
\NewDocumentCommand{\pLLMper}{ O{i} }{{{\mathfrak{L}\Pi}^{#1}}}
\NewDocumentCommand{\SCHUMAKER}{ O{i} }{\widehat{\mathfrak{L}\Pi}^{#1}}

\newcommand{\vertiii}[1]{{\left\vert\kern-0.25ex\left\vert\kern-0.25ex\left\vert #1 
    \right\vert\kern-0.25ex\right\vert\kern-0.25ex\right\vert}}

\usepackage{lipsum}
\usepackage{amsfonts}
\usepackage{graphicx}
\usepackage{epstopdf}
\usepackage{algorithmic}
\ifpdf
  \DeclareGraphicsExtensions{.eps,.pdf,.png,.jpg}
\else
  \DeclareGraphicsExtensions{.eps}
\fi

\newsiamremark{hypothesis}{Hypothesis}
\crefname{hypothesis}{Hypothesis}{Hypotheses}
\newsiamthm{claim}{Claim}

\headers{Energy conservative isogeometric methods for the wave equation}{A.Bressan, A. Buffa, A. Kushova, and  R. Vázquez}

\title{Energy conservative isogeometric methods for the wave equation\thanks{Submitted to the editors DATE.
}}

\author{Andrea Bressan\thanks{IMATI-CNR Pavia (\email{andrea.bressan@imati.cnr.it}).} \and Annalisa Buffa\thanks{École Polytechnique Fédérale de Lausanne  (\email{annalisa.buffa@epfl.ch}, \email{rafael.vazquez@epfl.ch}).}
\and Alen Kushova\thanks{Università degli studi di Pavia (\email{alen.kushova01@universitadipavia.it}).}\and Rafael Vázquez \footnotemark[3]}

\usepackage{amsopn}

\makeatletter
\newcommand*{\addFileDependency}[1]{
  \typeout{(#1)}
  \@addtofilelist{#1}
  \IfFileExists{#1}{}{\typeout{No file #1.}}
}
\makeatother

\newcommand*{\myexternaldocument}[1]{%
    \externaldocument{#1}%
    \addFileDependency{#1.tex}%
    \addFileDependency{#1.aux}%
}

\ifpdf
\hypersetup{
  pdftitle={Energy conservative isogeometric techniques for the wave equation},
  pdfauthor={A. Bressan, A. Buffa, A. Kushova, and  R. Vázquez}
}
\fi

\myexternaldocument{ex_supplement}

\begin{document}

\maketitle

\begin{abstract}
  We analyze the wave equation in mixed form, with periodic and/or 
  Dirichlet homogeneous boundary conditions, and nonconstant coefficients 
  that depend on the spatial variable. For the discretization, the weak 
  form of the second equation is replaced by a strong form, written in 
  terms of a projection operator. The system of equations is discretized 
  with B-splines forming a De Rham complex along with suitable commutative 
  projectors for the approximation of the second equation. The discrete 
  scheme is energy conservative when discretized in time with a 
  conservative method such as Crank-Nicolson. 
  We propose a convergence analysis of the method to study the dependence 
  with respect to the mesh size $h$, with focus on the consistency error. 
  Numerical results show optimal convergence of the error in energy norm, 
  and a relative error in energy conservation for long-time simulations of 
  the order of machine precision. 
\end{abstract}



\section{Introduction}
\label{sec:introduction}
The purpose of this paper is to provide a numerical analysis of 
discretization schemes applied to the anisotropic wave equation in mixed form, 
in which one of the variational equations is replaced by a suitably 
projected equation within the discrete space, introducing a consistency error. 
The motivation for our work stems from analogous discretization techniques recently proposed for 
Maxwell's equations, specifically in the long-term study and simulation of 
particle dynamics in plasma, as well as for the electrodynamics of phenomena 
resulting from particle motion \cite{kraus2017gempic}, \cite{holderied2021mhd}. 
The simplified model problem of the wave equation allows us to analyze the proposed method,
paying particular attention to the introduced consistency error.

The wave equation has been studied extensively in theory 
and numerical approximations. Existence and uniqueness
results are well known in literature \cite{lions2000evolution},\cite{evans10}. 
The present work considers the wave equation as a first order 
hyperbolic system, introducing velocity and pressure fields following 
approaches found in prior works \cite{becache2000analysis}, \cite{glowinski2004solution}, \cite{boffi2013convergence}.
For the resulting variational formulations in mixed form 
it is acknowledged that the successful approximation of the problem 
requires suitable compatibility conditions within the involved 
discrete spaces \cite{boffi2013mixed}. In general, mixed methods for the wave equation consider 
the discretization of vector fields in some $\mathbf{H}(\textup{div})$-conforming spaces while 
scalar fields in some $L^2$ spaces. To meet these conditions, we construct discrete spaces that 
adhere to the De Rham complex of exterior calculus \cite{arnold2006finite}, 
within the framework of isogeometric analysis \cite{hughes2005isogeometric}. 

Briefly, Isogeometric Analysis (IgA) utilizes spline functions, or their generalizations, 
for both representing the computational domain and approximating solutions to 
the partial differential equation that models the relevant problem. This approach aims to 
streamline the interoperability between computer-aided design and numerical simulations.
Moreover, IgA derives advantages from the approximation properties of splines, 
where their high continuity contributes to enhanced accuracy compared to $C^0$ piecewise polynomials. 
This characteristic is well-documented in literature, \cite{Evans_Bazilevs_Babuska_Hughes,bressan2018approximation,Sangalli2018}.
As regards the De Rham complex for tensor-product B-splines, initially introduced and analyzed by \cite{buffa2011isogeometric}, 
it has found wide applications in Galerkin approximation of Maxwell's equations \cite{ratnani2012arbitrary},
and divergence-free methods for incompressible fluid flow \cite{evans2013darcy},
\cite{evans2013steady},\cite{evans2013unsteady}. 

In this paper, together with the discrete B-spline spaces of the De Rham complex, 
we build quasi-interpolant projections that commute with the divergence operator. 
Various families of quasi-interpolant operators have been defined and presented in the context of 
spline approximation, \cite{de1973spline},\cite{lyche1975local}. 
Here, we implement a local quasi-interpolant operator, as presented in \cite{lee2000some},
wherein explicit formulas are provided for computing the coefficients of the projection. 
These explicit formulas entail pointwise evaluation of the function to be projected. 
Employing this operator, we project the second equation of the variational formulation, 
reducing the problem to a single equation for the velocity. Subsequently, the pressure field is 
determined after the computation of velocity. The resulting semi-discrete problem preserves 
the total energy of the system, as for a standard Galerkin method. 

Given our interest in preserving energy for long time simulations, 
the choice of an energy-preserving method in time is mandatory. 
We adopt Crank-Nicolson which is a second order energy 
conservative time discretization. Other time discretization methods that 
conserve energy can be applied with appropriate modifications 
to the fully discrete system. The numerical tests show a relative error 
in energy conservation for long time simulations of the order of machine 
precision.

Finally, error estimates in energy norm for Galerkin discretizations of mixed 
formulations are well known in literature \cite{boffi2013mixed}, 
and in particular for the wave equation in mixed form \cite{boffi2013convergence}.
Here, as the main contribution of this work, we present an error convergence analysis for our method with a 
generic family of projections, focusing on the consistency error. 
By assuming good approximation properties of the projections, 
together with its formal adjoint operator, we proof high-order 
convergence with respect to the mesh size $h$. Since the specific 
quasi-interpolant that we tested numerically does not fullfil the 
requires on the adjoint operator, we relax the assumption and proof 
that the numerical scheme converges linearly with respect to the 
mesh size. Numerical simulations confirm optimal high-order convergence 
for the implemented quasi-interpolant, similar to a standard Galerkin scheme,
which suggests an improvement of our theoretical results can be achieved. 
The global convergence of the scheme is of second order, as expected from the use of Crank-Nicolson method. 

The paper's structure is organized as follows. In section \ref{sec:wave} we present the 
model problem. In section \ref{sec:bsplines} we provide a brief overview of the isogeometric framework.
Section \ref{sec:discretization} introduces the discretization spaces along with commutative projections 
and presents the discrete problem. In section \ref{sec:error_analysis} we present the a priori error 
estimates in energy norm. Implementation details are covered in section \ref{sec:implementation_and_numerical_results}, 
with numerical results presented in section \ref{sec:numerical_simulations_in_2d}, and finally, in section
\ref{sec:conclusions} we draw our conclusions.

\section{Wave equation}
\label{sec:wave}
We start presenting the wave equation and its formulation as a first order hyperbolic system.
\subsection{Strong formulation and boundary conditions}
Let the domain $\Omega \subset \mathbb{R}^d$, with $d=2,3$, and let the time domain interval $I = [0,T]$ with $T > 0$. We consider our model problem, the second-order scalar wave equation with space dependent coefficients, that reads like
\begin{equation}
    \begin{cases}
    \label{eqn:2.1}
        u_{tt}-\text{div}(c^2\grad{u})= 0 & \text{in }\Omega \times I,\\
        u(\cdot ,0) = u_0,\quad u_t(\cdot,0) = u_1  & \text{in } \Omega,
    \end{cases}
\end{equation}
with the subscript $\cdot _t$ indicating the partial derivative with respect to time, $u_0$ and $u_1$ are the initial conditions, and the coefficient $c$ is a positive, uniformly bounded and smooth scalar field in $\Omega$. The problem has to be completed with boundary conditions, that we will detail below. Following the idea in \cite{glowinski2004solution}, by introducing the new variables, $\mathbf{v}=c\grad{u}$ and $\phi = u_t$, we rewrite \eqref{eqn:2.1} as a first order hyperbolic system, in the form
\begin{equation}
\label{eqn:2.2} 
    \begin{cases}
        \mathbf{v}_t = c\grad{\phi} & \text{in } \Omega\times[0,T],\\
        \phi_t = \text{div}(c \mathbf{v}) &  \text{in } \Omega\times[0,T],\\
        \mathbf{v}(\cdot, 0) =c\grad{u_0},\quad \phi(\cdot ,0) =u_1 & \text{in } \Omega.
    \end{cases}
\end{equation}

We will discretize the problem with an isogeometric method, for which we will assume that the domain is given by a parameterization of the form $\Omega = \mathcal{F}(\Hat{\Omega})$, where $\Hat{\Omega} = [0,1]^d$ is the parametric domain, and $\mathcal{F}$ is the isogeometric map to be detailed in Section \ref{sec:bsplines}. We split the boundary of $\Omega $ in two parts, with mutually disjoint interiors, denoted $\Gamma_D = \mathcal{F}(\hat{\Gamma}_D) $ and $\Gamma_P = \mathcal{F}(\hat{\Gamma}_P)$, respectively corresponding to Dirichlet and periodic boundary sides. Moreover, we assume that periodicity occurs only in the last parametric direction, and split the periodic boundary into two parts, $\Gamma_{P,1}$ and $\Gamma_{P,2}$, such that $\Gamma_{P} = \Gamma_{P,1} \cup \Gamma_{P,2}$, where $\Gamma_{P,i} = \mathcal{F}(\hat{\Gamma}_{P,i})$ for $i=1,2$ with $\hat{\Gamma}_{P,1} = [0,1]^{d-1}\times \{0\}$ and $\hat{\Gamma}_{P,2} = [0,1]^{d-1}\times \{1\}$. An illustration of the isogeometric map and the split of the boundary is given in Figure \ref{fig:new_fig_1}. 
\begin{figure}
    \centering
    \begin{tikzpicture}
        \draw[->] (0,0) -- (3.5,0); 
        \draw[->] (0,0) -- (0,3.5); 
        \draw[very thick] (0,0) rectangle (2,2); 
        \draw[->] (7,0) -- (10.5,0); 
        \draw[->] (7,0) -- (7,3.5);  
        \draw[very thick] (8.5,0) arc (0:90:1.5); 
        \draw[very thick] (10,0) arc (0:90:3);    
        \draw[very thick] (8.5,0) -- (10,0); 
        \draw[very thick] (7,1.5) -- (7,3); 
        \draw (0 cm,1pt) -- (0 cm,-1pt) node[anchor=north] {$0$};
        \draw (2 cm,1pt) -- (2 cm,-1pt) node[anchor=north] {$1$};
        \draw (3.5 cm,1pt) node[anchor=south] {$\hat{x}_1$};
        \draw (7 cm,1pt) -- (7 cm,-1pt) node[anchor=north] {$0$};
        \draw (8.5 cm,1pt) -- (8.5 cm,-1pt) node[anchor=north] {$1$};
        \draw (10 cm,1pt) -- (10 cm,-1pt) node[anchor=north] {$2$};
        \draw (10.5 cm,1pt) node[anchor=south] {$x_1$};
        \draw (1pt,2 cm) -- (-1pt,2 cm) node[anchor=east] {$1$};
        \draw (1pt,3.5 cm) node[anchor=west] {$\hat{x}_2$};
        \draw (7cm+1pt,1.5 cm) -- (7cm-1pt,1.5 cm) node[anchor=east] {$1$};
        \draw (7cm+1pt,3 cm) -- (7cm-1pt,3 cm) node[anchor=east] {$2$};
        \draw (7cm+1pt,3.5 cm) node[anchor=west] {$x_2$};
        \draw[very thick] (1,1) node {$\hat{\Omega}$}
            (1,-0.4) node {$\hat{\Gamma}_{P,1}$} 
            (1,2.3) node {$\hat{\Gamma}_{P,2}$} 
            (-0.3,1) node {$\hat{\Gamma}_{D}$} 
            (2.3,1) node {$\hat{\Gamma}_{D}$} 
            (8.8,1) node {$\Omega$}; 
        \draw[->] (9.25,0) -- (9.25,-1);
        \draw[->] (7,2.25) -- (6,2.25);
        \draw (9.55,-0.5) node {$\mathbf{n}_1$}
            (9.25,0.2) node {$\Gamma_{P,1}$}
            (9.25,2.4) node {$\Gamma_{D}$}
            (6.5,2.45) node {$\mathbf{n}_2$} 
            (7.35,2.25) node {$\Gamma_{P,2}$} 
            (7.75,1) node {$\Gamma_{D}$};
        \draw[->,thick] (3,2) ..controls (3.9,2.4) and (4.6,2.5) .. (5.5,2.5);
        \draw[very thick] (4.5,2.6) node {$\mathcal{F}$};
    \end{tikzpicture}
    \caption{Isogeometric map for the quarter of a bidimensional ring $\Omega$ with mixed Dirichlet and periodic boundaries.}
    \label{fig:new_fig_1}
\end{figure}
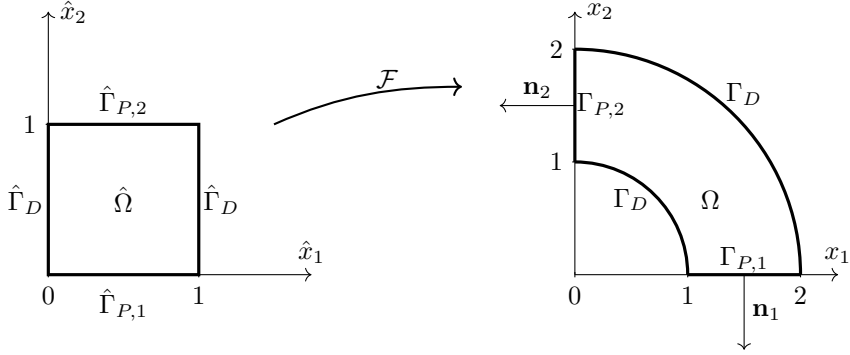
Then, we impose the Dirichlet boundary condition
\begin{equation*}
        u = g \qquad \text{on } \Gamma_D \times I,
\end{equation*}
{ where we assume for simplicity that $g$ does not depend on the time $t$.}
To impose the periodic boundary conditions, we introduce for scalar fields the trace operator $\gamma:H^1(\Omega) \to H^{\frac{1}{2}}(\partial \Omega)$, $\gamma: u \mapsto u\vert_{\partial \Omega} $, and for vector fields the normal trace operator ${\gamma}_{n}: H(\text{div};\Omega) \to H^{-\frac{1}{2}}(\partial \Omega)$, $\gamma_{n}:\mathbf{v}\mapsto \mathbf{v}\cdot\mathbf{n}$, with $\mathbf{n}$ the outgoing unit normal at $\partial\Omega$. When needed, we will restrict these operators to a part of the boundary, simply denoting this restriction as $\cdot\vert_{\Gamma_D}$ or $\cdot\vert_{\Gamma_{P,1}} $ and $ \cdot\vert_{\Gamma_{P,2}}$. With this notation, Dirichlet boundary condition reads like $\gamma(u)|_{\Gamma_D} = g$, for every $ t \in I$, and for simplicity we will assume that $g$ does not depend on $t$, while periodic boundary conditions read like, for every $t \in I$,
\begin{align*}
    \gamma(u)\big\vert_{\Gamma_{P,1}} = \gamma(u)\big\vert_{\Gamma_{P,2}},\quad & \quad  \gamma(u_t)\big\vert_{\Gamma_{P,1}} = \gamma(u_t)\big\vert_{\Gamma_{P,2}},\\
    {\gamma}_{n}(c\nabla u)\big\vert_{\Gamma_{P,1}} &= {\gamma}_{n}(c\nabla u)\big\vert_{\Gamma_{P,2}}. 
\end{align*}
In what follows, we make the assumption that $c$ is periodic, in the sense that its pull-back $\hat{c} = c \circ\mathcal{F}$, is periodic in the last parametric direction of $\hat \Omega$, that is, it satisfies periodicity conditions on $\hat{\Gamma}_P$. Therefore, the Dirichlet and periodic boundary conditions for the first order system \eqref{eqn:2.2} read as
\begin{align*}
&        \phi = 0 \quad \text{on } \Gamma_D \times I,\\
&    {\gamma}_{n}(\mathbf{v})\big\vert_{\Gamma_{P,1}} = {\gamma}_{n}(\mathbf{v})\big\vert_{\Gamma_{P,2}},\quad
    \gamma(\phi)\big\vert_{\Gamma_{P,1}} = \gamma(\phi)\big\vert_{\Gamma_{P,2}},\quad \text{ for all } t \in I.
\end{align*}
We remind that existence and uniqueness results for the solution $u$ are well known, see \cite[Section 7.2.2]{evans10}, while the regularity of the solution will depend on the regularity of the initial conditions \cite[Section 7.2.3]{evans10}. 

\subsection{Weak formulation and conservation of energy}
Let us define the following Hilbert spaces over $\Omega \subset \mathbb{R}^d$. $L^2(\Omega)$ is the usual Hilbert space 
of square integrable functions over $\Omega$, endowed with the classical $L^2$-norm $\|\cdot\|_{L^2(\Omega)}$. 
By $\mathbf{L}^2(\Omega)$ we denote its vectorial counterpart. The Hilbert spaces $H^k(\Omega)$ denote the functions 
in $L^2(\Omega)$ such that their $k$th-order derivatives also belong to $L^2(\Omega)$, and their vectorial counterparts 
will be denoted by $\mathbf{H}^k(\Omega)$. We define
\begin{align*}
    \mathbf{H} (c,\text{div};\Omega) &\coloneqq \{\mathbf{v} \in \mathbf{L}^2(\Omega): \text{div } (c\mathbf{v}) \in L^2 (\Omega) \},\\
    \mathbf{H}_P (c,\text{div};\Omega) &\coloneqq \{\mathbf{v} \in \mathbf{H} (c,\text{div};\Omega): {\gamma}_{n}(c\mathbf{v})|_{\Gamma_{P,1}} = {\gamma}_{n}(c\mathbf{v})|_{\Gamma_{P,2}}\},
\end{align*}
We will also make use of the space $\mathbf{H}^k(\textup{div};\Omega)$, the space of functions in $\mathbf{H}^k(\Omega)$ 
such that their divergence belongs to ${H}^k(\Omega)$.
We can then write the weak formulation of \eqref{eqn:2.2}, that is: find  $\mathbf{v}\in \mathbf{H}_P(c,\text{div};\Omega)$ and $\phi \in L^2(\Omega)$ such that the following equations hold
\begin{subequations}
\label{eqn:2.3and2.4}
\begin{align}
\label{eqn:2.3}
    \int_\Omega \mathbf{v}_t\cdot \mathbf{w}\ \mathrm{d}\mathbf{x} &= - \int_\Omega \text{div}(c\mathbf{w})\ \phi\ \mathrm{d}\mathbf{x} \quad &\forall \mathbf{w} \in\mathbf{H}_P (c,\text{div};\Omega), \\  
    \label{eqn:2.4}
    \int_\Omega \phi_t \psi\ \mathrm{d}\mathbf{x} &=  \int_\Omega \text{div}(c\mathbf{v})\ \psi\ \mathrm{d}\mathbf{x}, \quad &\forall \psi \in L^2 (\Omega).
\end{align}
\end{subequations}
Notice that adding equation \eqref{eqn:2.3} to equation \eqref{eqn:2.4} and choosing the test functions as $\mathbf{w}= \mathbf{v}$ and $\psi = \phi $, we obtain
\begin{equation*}
    \int_\Omega \mathbf{v}_t \cdot \mathbf{v} \mathrm{d}\mathbf{x} + \int_\Omega \phi_t\phi \mathrm{d}\mathbf{x} = 0,
\end{equation*}
and defining the quantity of total energy as
\begin{equation}
\label{eqn:2.5}
    E(t)\coloneqq \frac{1}{2} \int_\Omega |\mathbf{v}|^2 + |\phi|^2  \mathrm{d}\mathbf{x},
\end{equation}
it is obviously seen that $\frac{\mathrm{d}E(t)}{\mathrm{d}t}  = 0$, that is, the energy is preserved. 

{Building upon the work conducted by Kraus \textit{et al.} in the field of magnetohydrodynamics \cite{kraus2017gempic,holderied2021mhd}, 
we aim to introduce and analyze, for this simplified model problem, a numerical method that preserves the total energy of the system for long-term simulations,
and for which \eqref{eqn:2.3} is solved in a weak sense, while \eqref{eqn:2.4} is solved in strong form.}
To achieve this purpose, we introduce a generic linear operator
\begin{equation*}
    \mathcal{T}: L^2(\Omega) \to L^2(\Omega), 
\end{equation*}
and replace the second equation \eqref{eqn:2.4} by
\begin{equation}
\label{eqn:2.6}
    \phi_t = \mathcal{T} ( \text{div}(c\mathbf{v})).
\end{equation}
Note that we recover \eqref{eqn:2.4} if $\mathcal{T}$ is the identity operator. This change in the equation affects the conservation of the total energy of the system. In order to retrieve this conservation we replace \eqref{eqn:2.3} with
\begin{equation*}
        \int_{\Omega} \mathbf{v}_t\cdot\mathbf{w}\ \mathrm{d}\mathbf{x} = -\int_{\Omega} \mathcal{T}(\text{div}(c\mathbf{w}))\phi\ \mathrm{d}\mathbf{x}.
\end{equation*}
The modified problem that we obtain, with mixed homogeneous Dirichlet and periodic boundary conditions, reads: find $\mathbf{v}\in \mathbf{H}_P(c,\text{div};\Omega)$ and $\phi \in L^2(\Omega)$ such that the following equations hold 
\begin{subequations} \label{eq:problem_mixed_form_modified}
\begin{equation}
    \label{eqn:2.7}
    \int_{\Omega} \mathbf{v}_t\cdot\mathbf{w}\ \mathrm{d}\mathbf{x} = -\int_{\Omega} \mathcal{T}(\text{div}(c\mathbf{w}))\phi\ \mathrm{d}\mathbf{x}, \quad \forall \mathbf{w} \in\mathbf{H}_P(c,\text{div};\Omega),
\end{equation}
\begin{equation}
    \label{eqn:2.8}
    \phi_t = \mathcal{T} ( \text{div}(c\mathbf{v})).
\end{equation}
\end{subequations}
\section{B-splines and IGA framework}
\label{sec:bsplines}
In this section we recall the definition of B-splines in the univariate and 
multivariate context. We also {introduce the isogeometric spaces, 
together with their regularity assumptions.} 
\subsection{Univariate B-splines}
\label{sec:sub_univariate_bsplines}
For the definition of B-splines we follow the notation and the guidelines of 
\cite{inbook2016Buffa} and \cite{lyche2008spline}. Let us introduce a knot vector 
$\Xi = \{\xi_1,\dots , \xi_{n+p+1}\}$, where $\xi_i\in \mathbb{R}$ is the $i$-th knot, 
$p$ is the polynomial degree and $n$ is the number of basis functions. We assume that 
$ \xi_1 \leq \xi_2 \leq \dots \leq \xi_{n+p} \leq \xi_{n+p+1},$ and without loss of 
generality $\xi_1 = 0,\ \xi_{n+p+1} = 1$. We also admit at most $p+1$ repeated knots.
We introduce the B-spline functions $\{\hat{B}_{i,p}\}_{i=1}^{n}$ of degree $p$ over 
the knot vector $\Xi$ following the Cox-DeBoor recursive formula \cite[Chapter IX]{de1978practical}. 
We obtain a set of $n$ B-splines with the following properties: non-negativity, 
partition of unity and local support of each basis function, see \cite{inbook2016Buffa}.
The univariate B-spline space of degree $p$ over the knot vector $\Xi$ is:
\begin{equation*}
    \hat{S}_p(\Xi) \coloneqq \text{span}\{\hat{B}_{i,p}: i = 1,\dots,n \}.
\end{equation*}
We introduce also the vector $\mathbf{Z}= \{\zeta_1,\dots,\zeta_{z}\}$, 
of knots without repetitions, which are also called breakpoints, and denote with 
$m_j$ the multiplicity of $\zeta_j$, such that $\sum_{j = 1 }^{z} m_j = n+p+1$, 
and $\zeta_1 = \xi_1 = \dots = \xi_{m_1}$, $\zeta_2 = \xi_{m_1+1} = \dots = \xi_{m_1 + m_2} $ 
and so on. Notice that
\begin{equation*}
    \Xi = \{\underbrace{\zeta_1, \dots, \zeta_{1}}_{m_1 \text{ times}},\underbrace{\zeta_2, \dots, \zeta_{2}}_{m_2 \text{ times}},\dots,\underbrace{\zeta_{z}, \dots, \zeta_{z}}_{m_{z} \text{ times}}\},
\end{equation*}
with each $\zeta_j$ repeated $m_j$ times, and $1\leq m_j \leq p+1 $ for all internal 
knots. { The breakpoints form a partition of the unit interval, and we say the 
partition is locally quasi-uniform if there exists a constant $\beta \geq 1$ independent of $z$ such that 
$$
\beta^{-1} \leq \dfrac{\zeta_{j+1} -\zeta_{j}}{\zeta_{j} - \zeta_{j-1}}\leq \beta, \quad \forall j = 2,\dots,z-1.
$$
}

Assuming that the multiplicity of internal knots is at most $p$, we can write the derivative of a B-spline as follows 
\begin{equation}
\label{eqn:3.1}
    \frac{\partial \hat{B}_{i,p}}{\partial \xi}(\xi) =  \hat{D}_{i-1,p-1}(\xi) - \hat{D}_{i,p-1}(\xi),
\end{equation}
with the Curry-Schoenberg spline basis: 
\begin{equation}
    \label{eqn:3.1.1}
        \hat{D}_{i,p-1}(\xi) = \frac{p}{\xi_{i+p+1}-\xi_{i+1}} \hat{B}_{i+1,p-1}(\xi),\quad \text{for } i = 1,\dots ,n-1.
    \end{equation}    
where we assumed that $\hat{D}_{1,p-1}(\xi)=\hat{D}_{n+1,p-1}(\xi)=0$.
Note that the derivative belongs to the spline space 
$\hat{S}_{p-1}(\Xi') $ where $\Xi' = \{\xi_2,\dots,\xi_{n+p}\}$. 

In order to handle periodic boundary conditions, we follow 
the construction of periodic B-spline spaces, see for instance 
\cite{montardini2018isogeometric}. Consider a uniformly spaced 
\textit{closed} knot vector $\Xi$ of degree $p$, { that is $\xi_1 < \dots < \xi_{p+1}=0$ and $1=\xi_{n+1} < \dots < \xi_{n+p+1}$ with at least 
$p$ internal knots.} We can construct a 
periodic basis on such knot vector, by identifying the basis 
functions of the B-spline space $\hat{S}_{p}(\Xi)$ in this way:
\begin{equation*}
    \begin{cases}
        \hat{B}_{i,p}^{Per} \coloneqq \hat{B}_{i,p} + \hat{B}_{n-p+i,p}, &\text{for } i = 1, \dots ,p;\\
        \hat{B}_{i,p}^{Per} = \hat{B}_{i,p}, &\text{otherwise.}
    \end{cases}
\end{equation*}
\begin{figure}
    \centering
    \includegraphics[width=12cm,height=4cm]{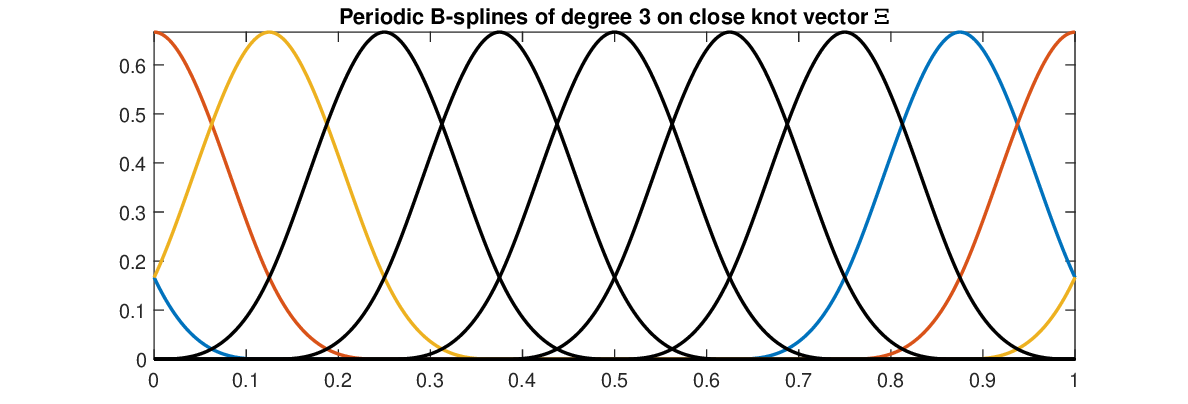}
    \caption{B-spline basis of degree 3 in the periodic case. In color the first three B-spline basis. The left hand side matches with regularity $C^2$ with their         right hand side. }
    \label{fig:1new}
\end{figure}
Here we are gluing together the first $p$ basis functions with the last $p$, this way we get continuity of derivatives up to order $p-1$. In Figure \ref{fig:1new} is shown the behavior of this periodic basis at the boundary. We can introduce the periodic B-spline space with highest regularity at the boundary as: 
\begin{equation*}
      \hat{S}_{p}^{Per}(\Xi) \coloneqq \text{span}\{\hat{B}_{i,p}^{Per}: i = 1, \dots ,n-p\}.
\end{equation*}
It is easy to see that the dimension of the spline space $\hat{S}_{p}^{Per}(\Xi)$ is $n-p$, which is the same of $\hat{S}_{p-1}^{Per}(\Xi')$. Finally we note that equation \eqref{eqn:3.1} holds for periodic B-splines too. For more information and properties on B-splines, we refer the reader to \cite{de1978practical}. 

\subsection{Multivariate B-splines}
\label{sec:sub_multivariate_bsplines}
Multivariate B-splines are introduced and defined by tensor product, starting from univariate B-splines on each spatial direction.
Let $d$ be the space dimension, usually $d=2,3$, and assume $n_l \in \mathbb{N} $ the number of univariate basis functions in direction $l$, $p_l \in \mathbb{N}$ is the degree, while $\Xi_l = \{\xi_{l,1}, \dots, \xi_{l,n_l+p_l+1}\}$ and $\mathbf{Z}_l = \{\zeta_{l,1}, \dots, \zeta_{l,z_l}\}$ are respectively the knots and breakpoints in direction $l$. We also set the polynomial degree vector $\mathbf{p} = (p_1, \dots , p_d)$ and $\mathbf{\Xi} = \{\Xi_1 , \dots ,\Xi_d \}$. Finally we set $\mathbf{J} = \{ \mathbf{j} = (j_1,\dots,j_d) \subset \mathbb{N}^d : 1 \leq j_l \leq n_l \}$. The set of multivariate B-splines basis functions of degree vector $\mathbf{p}$ is
\begin{equation*}
    \{ \hat{B}_{\mathbf{i,p}}(\boldsymbol{\xi}) = \hat{B}_{i_1,p_1}(\xi_1) \dots \hat{B}_{i_d,p_d}(\xi_d), \text{for } \mathbf{i} \in \mathbf{J}\},
\end{equation*}
and the B-spline multivariate space of degree vector $\mathbf{p}$ over $\mathbf{\Xi}$ is 
\begin{equation*}
    \hat{S}_{\mathbf{p}}(\mathbf{\Xi} ) \coloneqq \text{span}\{ \hat{B}_{\mathbf{i,p}}(\boldsymbol{\xi}):\mathbf{i}\in \mathbf{J}\} = \otimes_{l=1}^d \hat{S}_{p_l}(\Xi_l).
\end{equation*}
A similar construction also applies for multivariate periodic B-splines. It is also possible to combine tensor product of standard B-splines in the first directions with periodic ones for the last direction, which is the case of the mixed boundary condition setting for our model problem. 

\begin{figure}
    \centering
    \begin{tikzpicture}
        \draw[very thick] (0,0) rectangle (3,3); 
        \draw[step=1cm] (0,0) grid (3,3); 
        \draw[very thick] (6.5,0) arc (0:90:1.5); 
        \draw (7,0) arc (0:90:2); 
        \draw (7.5,0) arc (0:90:2.5); 
        \draw[very thick] (8,0) arc (0:90:3);
        \draw[very thick] (6.5,0) -- (8,0); 
        \draw[very thick] (5,1.5) -- (5,3); 
        \draw (6.2990,0.750) -- (7.5981,1.5);
        \draw (5.75,1.2990) -- (6.5,2.5981);
        \filldraw[fill=green!20!white, draw=green!50!black] (1,1) rectangle (2,2);
        \filldraw[fill=green!20!white, draw=green!50!black] (6.7321,1) -- (7.1651,1.25) arc (30:60:2.5) -- (6,1.7321) arc (60:30:2); 
        \draw[very thick] (1.5,1.5) node {$\mathbf{Q}$}
        (6.6,1.6) node {$\mathbf{K}$}; 
        \draw[->,thick] (3.5,1) -- (5,1);
        \draw[very thick] (4,1.5) node {$\mathcal{F}$};
    \end{tikzpicture}
    \caption{Mesh $\widehat{\mathcal{M}}$ in the parametric domain, and its image $\mathcal{M}$ on the physical domain.}
    \label{nuova}
    \end{figure}
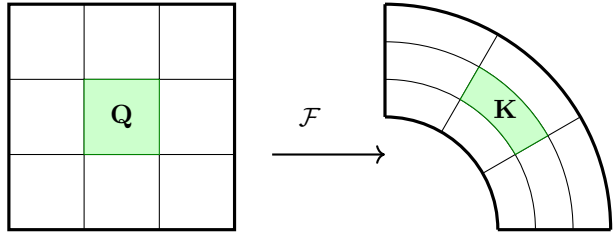
    
\subsection{Isogeometric analysis framework}
\label{sec:sub_isogeometric_analysis_framework}
The isogeometric map $\mathcal{F}: \hat{\Omega} \to \Omega$, is a parameterization of the geometry of the physical domain, based on B-splines (or more often in their rational counterpart of NURBS), usually indicated by 
\begin{equation*}
    \mathcal{F} \coloneqq \sum_{\mathbf{i}\in \mathbf{J}} \mathbf{c_i} \hat{B}_{\mathbf{i,p}}.
\end{equation*}
For the analysis of the discrete problem, we need to introduce 
some assumptions on $\mathcal{F}$. First, we introduce the parametric Bézier mesh $\widehat{\mathcal{M}}$, which is 
\begin{equation*}
    \widehat{\mathcal{M}} \coloneqq \{\mathbf{Q_j} = I_{1,j_1} \times \dots \times I_{d,j_d} : I_{l,j_l} = (\zeta_{l,j_l},\zeta_{l,j_l+1}), \text{ for } 1 \leq j_l \leq z_l-1\}.
\end{equation*}
{ Given an element $\mathbf{Q} \in \widehat{\mathcal{M}}$, we set $h_{\mathbf{Q}} = \textup{diam}(\mathbf{Q}_{\mathbf{j}})$, while $h = \max\{h_{\mathbf{Q} },\ \mathbf{Q}\in \widehat{\mathcal{M}}\}$.
Moreover, given $\mathbf{D} \subset \hat{\Omega}$, we denote by $\widetilde{{\mathbf{D}}}$ its support extension, that is the interior of the union of the supports of basis functions whose support intersects ${\mathbf{D}}$. }

The Bézier mesh is defined as the image of elements in 
$\widehat{\mathcal{M}}$ through $\mathcal{F}$:
\begin{equation*}
    \mathcal{M}\coloneqq \{ \mathbf{K}\subset \Omega : \mathbf{K} = \mathcal{F}(\mathbf{Q}), \mathbf{Q}\in \widehat{\mathcal{M}}\},
\end{equation*}
see Figure \ref{nuova}. 
{ We have $\widetilde{\mathcal{F}(\mathbf{D})} = \mathcal{F}(\widetilde{\mathbf{D}})$, for every subset $\mathbf{D} \subset \hat{\Omega}$, denoting the support extension in the physical domain. }
The first assumption is the following.
\begin{ass}\label{assumption_1}
We assume that $\mathcal{F}$ is a bi-Lipschitz homeomorphism. 
Moreover, $\mathcal{F}|_{\overline{\mathbf{Q}}}$ belongs to 
$\mathcal{C}^\infty(\overline{\mathbf{Q}})$ for all 
$\mathbf{Q}\in \widehat{\mathcal{M}}$, where 
$\overline{\mathbf{Q}}$ denotes the closure of $\mathbf{Q}$, 
and $\mathcal{F}^{-1}|_{\overline{\mathbf{K}}}$ belongs to 
$\mathcal{C}^\infty(\overline{\mathbf{K}})$, for all 
$\mathbf{K} \in \mathcal{M}$.
\end{ass}
This prevents the existence of singularities and self-intersections in the parameterization $\mathcal{F}$. 
The second assumption simplifies the dealing with the boundary. 
\begin{ass}\label{assumption_2}
The boundary region $\Gamma_D \subset \partial \Omega$ is the union of full faces of the boundary. More precisely, $\Gamma_D = \mathcal{F}(\hat{\Gamma}_D)$, with $\hat{\Gamma}_D$ a collection of full faces of the parametric domain $\hat{\Omega}$. Finally, if $\Gamma_P \neq \emptyset$, we assume the periodic boundary to be: $$ \Gamma_P = \mathcal{F}\big([0,1]^{d-1}\times\{0\}\big) \cup \mathcal{F}\big([0,1]^{d-1}\times\{1\}\big).$$
\end{ass}
{ Finally, we assume local quasi-uniformity of the univariate partitions, thus the parametric Bézier mesh is shape regular, that is
the ratio between the smallest edge of $\mathbf{Q}\in \widehat{\mathcal{M}}$ and its diameter $h_{\mathbf{Q}}$ is bounded uniformly with respect to 
$h$ and $\mathbf{Q}$. Analogously, the Bézier mesh is shape regular thanks to Assumption \ref{assumption_1}. 
}

\section{Discretization}
\label{sec:discretization}
{ In this section, our focus is on introducing the discretization of problem \eqref{eq:problem_mixed_form_modified}. 
We approximate $\mathcal{T}$ with quasi-interpolant projections that commute with the divergence operator. 
The Crank-Nicolson method is employed for temporal discretization, chosen for its conservativity. 
We emphasize that our objective is to study the approximation properties of the proposed method, both 
theoretically and numerically. To achieve this, we intentionally maintain a highly generic notation.
In \cref{sec:discretization_in_space_with commutative projections}, we present the spatial semi-discretization 
through the construction of isogeometric discrete spaces and the associated projections.
In \cref{sec:sub_quasi_interpolant_projections} we discuss a particularization of such projections, 
which are quasi-interpolant operators based on point evaluation of the functions to be projected.
This quasi-interpolant operators will be subjected to numerical study and testing within the proposed discretization framework.  
Finally, in \cref{sec:sub_conservative_semi_discretization_in_time} we recall the Crank-Nicolson semi discretization in time and 
the fully discrete problem is presented. 
}
\subsection{Discretization in space with commutative projections}
\label{sec:discretization_in_space_with commutative projections}
Here we introduce the discrete spaces for the mixed formulation given by equations \eqref{eqn:2.7} and \eqref{eqn:2.8}. 
We present the case $d = 2$, while the case $d = 3$ is completely analogous.
\subsubsection{Discrete spline spaces}
Let us standardize the notation following 
\cite{da2014mathematical}, and start with the case of Dirichlet boundary conditions, that is $\Gamma_P = \emptyset$, we define the spaces:
\begin{equation*}
    { X^1 \coloneqq \mathbf{H}(c,\text{div};\Omega)}, \quad X^2\coloneqq L^2(\Omega),
    \quad { \hat{X}^1 \coloneqq \mathbf{H}(\hat{c},\text{div};\hat{\Omega})},\quad \hat{X}^2 \coloneqq L^2(\hat{\Omega}). 
\end{equation*}
Thanks to Assumption \ref{assumption_1}, which states that both $\mathcal{F}$ and its inverse are smooth, we can define the pull-backs that relate
these spaces as (see \cite[Section 2.2]{hiptmair2002finite})
\begin{align*}
    \iota^1(\mathbf{f}) &\coloneqq \det(D\mathcal{F})(D\mathcal{F})^{-1} (\mathbf{f}\circ \mathcal{F}), &\mathbf{f} \in X^1,\\
    \iota^2(f) &\coloneqq \det(D\mathcal{F}) (f\circ \mathcal{F}), &f \in X^2,
\end{align*}
where $D\mathcal{F}$ is the Jacobian matrix of $\mathcal{F}$. Then, due to the divergence preserving property of the map $\iota^1$, 
see \cite[Section 3.9]{monk2003finite}, we have $\text{div} \circ \iota^1 = \iota^2 \circ \text{div}$. 
Fixed a polynomial degree vector $\mathbf{p}=(p_1,p_2)$ and $\mathbf{\Xi}= (\Xi_1,\Xi_2)$, we define the discrete spaces on the 
parametric domain as:
\begin{align*}
    \hat{X}^1_h &\coloneqq \hat{S}_{p_1,p_2-1}(\Xi_1,\Xi_2')\times \hat{S}_{p_1-1,p_2}(\Xi_1',\Xi_2),\\
    \hat{X}^2_h &\coloneqq \hat{S}_{p_1-1,p_2-1}(\Xi_1',\Xi_2').
\end{align*}
The choice of the bases follows from \cite[Section 5.2]{da2014mathematical} and \cite[Section 4]{ratnani2012arbitrary}, that is
\begin{equation*}
    \hat{X}^1_h = span\{ \mathcal{J}_1 \cup \mathcal{J}_2 \},
\end{equation*}
where we set 
\begin{align*}
     \mathcal{J}_1 &= \{ \hat{B}_{i_1,p_1}(\xi_1)\hat{D}_{i_2,p_2-1}(\xi_2)\mathbf{e_1} :\ 1\leq i_1 \leq n_1,\ 1\leq i_2 \leq n_2-1\},\\
     \mathcal{J}_2 &= \{ \hat{D}_{i_1,p_1-1}(\xi_1)\hat{B}_{i_2,p_2}(\xi_2)\mathbf{e_2} :\ 1\leq i_1 \leq n_1-1,\ 1\leq i_2 \leq n_2\},
\end{align*}
and $\{\mathbf{e_1},\mathbf{e_2}\} $ is the canonical basis of $\mathbb{R}^2$. As regards the second discrete space, we set
\begin{equation*}
    \hat{X}^2_h = span\{\hat{D}_{i_1,p_1-1}(\xi_1) \hat{D}_{i_2,p_2-1}(\xi_2):\ 1\leq i_l \leq n_l-1,\ l=1,2 \}.
\end{equation*}
The discrete spaces on the physical domain $\Omega$ can be defined from the spaces $\hat{X}^1_h$ and $\hat{X}^2_h$ by push-forward,
that is, the inverse of the transformations $\iota^1$ and $\iota^2$, that commute with the divergence operator.
We have the following definitions:
\begin{align*}
    X^1_h \coloneqq \{\mathbf{f_h} : \iota^1(\mathbf{f}_h) \in \hat{X}^1_h \},\\
    X^2_h \coloneqq \{f_h : \iota^2(f_h) \in \hat{X}^2_h \}.
\end{align*}
For later use, we introduce a suitable notation for the basis of these discrete spaces. 
We denote by $\{ \mathbf{b}_{i,h} \}_{i = 1 }^{N}$ and $\{ \hat{\mathbf{b}}_{i,h} \}_{i = 1 }^{{N}}$ the sets of basis functions of 
$X^1_h$ and $\hat{X}^1_h$ respectively, reordered with lexicographic ordering, such that 
$\iota^1 (\mathbf{b}_{i,h}) = \hat{\mathbf{b}}_{i,h}$ for any $i = 1,\dots,{N},$ where $N = n_1(n_2-1) + (n_1 -1) n_2$. 
Analogously we can introduce the notations $\{b_{i,h}\}_{i = 1}^{M}$ and $\{\hat{b}_{i,h}\}_{i = 1}^{M}$ for the basis functions of the 
spaces $X^2_h$ and $ \hat{X}^2_h$ respectively, such that, for $i = 1,\dots, M$ with $M= (n_1 -1 )(n_2 - 1) $ we have 
$\iota^2(b_{i,h}) = \hat{b}_{i,h}$.

\subsubsection{Univariate projections}
\label{sec:sub_projections_in_the_discrete_spaces}
In order to discretize equations \eqref{eqn:2.7} and \eqref{eqn:2.8}, 
We follow the idea in \cite{holderied2021mhd} and in \cite{kushova2020master}, 
and approximate the operator $\mathcal{T}$ with a projector 
into the discrete spline space. We first introduce commutative projectors 
in the univariate case, which can be defined by a dual basis, 
i.e. $\hat{\pi}_{p,\Xi} : L^2(0,1) \longrightarrow \hat{S}_{p}(\Xi)$ such that
\begin{equation}
\label{eqn:3.2}
    \hat{\pi}_{p,\Xi}(f) = \sum_{i=1}^n \lambda_{i,p}(f) \hat{B}_{i,p}
\end{equation}
where $\lambda_{i,p}$ are a set of dual functionals verifying $\lambda_{i,p}(\hat{B}_{j,p}) = \delta_{ij}$, where $\delta_{ij}$ is the 
Kronecker delta. Notice that, with this property, the operator is a projection, that is 
$\hat{\pi}_{p,\Xi}(s) = s$, for all $s \in \hat{S}_{p}(\Xi) $.

We also recall the construction of univariate projections that 
commute with the derivative as in \cite[Section 3.1.2]{buffa2011isogeometric}.
In order to obtain the one-dimensional commuting diagram, we want the 
commutative projection $\hat{\pi}^c_{p-1,\Xi'}:L^2(0,1) \to S_{p-1}(\Xi') $ to 
satisfy:
\begin{equation*}
    \hat{\pi}_{p-1,\Xi'}^c \frac{\partial}{\partial \xi} f =  \frac{\partial}{\partial \xi} \hat{\pi}_{p,\Xi} f, 
\end{equation*}
for all $f$ in $H^1(0,1)$. In order to satisfy the previous equation, given 
$\hat{\pi}_{p,\Xi}$ as in \eqref{eqn:3.2}, its commutative projection is defined as
\begin{equation}
\label{eqn:3.4}
    \hat{\pi}_{p-1,\Xi'}^c g (\xi) \coloneqq \frac{\partial}{\partial \xi} \hat{\pi}_{p,\Xi} \int_0^{\xi} g(s) \mathrm{d}s, 
\end{equation}
for all functions $g$ such that $G(\xi) = \int_0^{\xi}g(s) \mathrm{d}s$ is in the domain of $\hat{\pi}_{p,\Xi}$.

To apply periodic boundary conditions we need to 
define the commutative projector in a different way, to ensure that
integrating a periodic function in $L^2(0,1)$, 
we obtain a periodic function in $H^1(0,1)$. Hence, given 
a periodic projection $\hat{\pi}_{p,\Xi}^{Per}$, we define its 
commutative one as follows 
\begin{equation}
\label{eqn:3.5}
    \hat{\pi}_{p-1, \Xi'}^{c, Per} g (\xi) \coloneqq 
    \bigg(\frac{\partial}{\partial \xi} \hat{\pi}_{p,\Xi}^{Per} \int_0^{\xi} \tilde{g}(s) \mathrm{d}s \bigg) \oplus \overline{g},
\end{equation}
for all functions $g = \tilde{g} \oplus \overline{g} \in L^2(0,1)$ such that $ \overline{g} \coloneqq \int_0^1 g(s) \mathrm{d}s \in \mathbb{R}$ and $\tilde{g} \in L^2_0(0,1)$, the space of functions in $L^2(0,1)$ with zero average. Notice that 
$G(\xi) = \int_0^{\xi}{ \tilde{g}}(s) \mathrm{d}s$ is in the domain of $\hat{\pi}_{p,\Xi}^{Per}$, and again we have
\begin{equation}
    \hat{\pi}_{p-1,\Xi'}^{c, Per} \frac{\partial}{\partial \xi} f =  \frac{\partial}{\partial \xi} \hat{\pi}_{p, \Xi}^{Per} f ,
\end{equation}
for all $f$ in the domain of definition of $\hat{\pi}_{p, \Xi}^{Per}$. It is easy to see that the commutative projections defined in this way
preserve splines.

\subsubsection{Multivariate construction}
The univariate projection operators introduced can be extended to the multidimensional case by tensor 
product constructions. Given $\hat{\Omega} = [0,1]^d \subset \mathbb{R}^d$, 
for $i = 1,2,\dots,d $, let us denote with $\hat{\pi}_{p_i,\Xi_i}$ 
a generic univariate projection as in \eqref{eqn:3.2}. We can 
define a multivariate projection as
\begin{equation}
    \hat{\Pi}_{\mathbf{p,\Xi}} \coloneqq \hat{\pi}_{p_1,\Xi_1} \otimes \hat{\pi}_{p_2,\Xi_2} \otimes \dots \otimes \hat{\pi}_{p_d,\Xi_d}.
\end{equation}
It is important to note that it can be expressed as 
\begin{equation}
    \hat{\Pi}_{\mathbf{p},\mathbf{\Xi}}(f) = \sum_{\mathbf{i}\in \mathbf{J}} \lambda_{\mathbf{i,p}}(f) \hat{B}_{\mathbf{i,p}},
\end{equation}
where $\mathbf{p}$, $\mathbf{\Xi}$ and $\mathbf{J}$ are given as in the multivariate B-spline construction, and each dual functional 
is defined from the univariate dual basis by the expression
\begin{equation}
    \lambda_{\mathbf{i,p}} = \lambda_{i_1,p_1} \otimes \lambda_{i_2,p_2} \otimes \dots \otimes \lambda_{i_d,p_d}.
\end{equation}
It is easy to see that the same constructions hold for the multivariate periodic case, or for combinations with periodic 
boundary conditions in only one parametric coordinate.

At this point we have all the ingredients to define projections on the spaces $\hat{X}^1_h $ and $\hat{X}^2_h $. The choice of the 
projection follows from the definition of the discrete spaces. More precisely we set 
\begin{align}
    \label{eqn:4.3}
    \hat{\Pi}^1 &= (\hat{\pi}_{p_1} \otimes \hat{\pi}_{p_2-1}^c) \times( \hat{\pi}_{p_1-1}^c\otimes \hat{\pi}_{p_2}),\\
    \label{eqn:4.4}
    \hat{\Pi}^2 &= \hat{\pi}_{p_1-1}^c \otimes \hat{\pi}_{p_2-1}^c.
\end{align}
where for simplicity of notation we have omitted the knot vector from the subscript of the univariate projectors.
These projectors satisfy the spline preserving properties, see \cite[Lemma 5.3]{da2014mathematical}.

The projectors into the discrete spline spaces in the physical domain are defined from 
the ones in the parametric domain \eqref{eqn:4.3} and \eqref{eqn:4.4}, and the corresponding pull-backs $\iota^1$ and $\iota^2$, 
in such a way they are uniquely characterized by the equations
\begin{equation}
    \label{eqn:4.6new}
    \begin{aligned}
        \iota^1(\Pi^1 \mathbf{f}) = \hat{\Pi}^1(\iota^1(\mathbf{f})),\\
        \iota^2(\Pi^2 f) = \hat{\Pi}^2(\iota^2(f)).
    \end{aligned}
\end{equation}
These projectors satisfy the commutativity property with the divergence operator, 
as stated in the following lemma:
\begin{lem}\label{lemma:commutativity of projections}
Given $\hat{c} $ and $c $ the coefficients of spaces $\hat{X}^1$ and $X^1$, the following equations hold
\begin{align}
\label{eqn:4.5}
    \textup{div }( \hat{\Pi}^{1}(\hat{c}\textbf{f})) &= \hat{\Pi}^2 ( \textup{div } (\hat{c}\textbf{f})), 
     \quad \forall \textbf{f} \in \hat{X}^1,\\
\label{eqn:4.8}
    \textup{div }( \Pi^{1}(c\textbf{f})) &= \Pi^2 ( \textup{div }(c \textbf{f})),\quad \forall \textbf{f} \in X^1.
\end{align}
\end{lem}
The proof of \eqref{eqn:4.5} is given in \cite[Lemma 5.5]{da2014mathematical}. Equation \eqref{eqn:4.8} is an immediate consequence 
of the definitions, together with the commutativity property \eqref{eqn:4.5}.
We conclude this paragraph with the following remark on the periodic case.
\begin{rmk}\label{rem:mixed_projections}
If $\Gamma_P \neq \emptyset$, we proceed analogously introducing the discrete spaces
\begin{align}
    \hat{X}^1_h &\coloneqq \Big(\hat{S}_{p_1}(\Xi_1)\otimes\hat{S}_{p_2-1}^{Per}(\Xi_2')\Big) \times 
     \Big( \hat{S}_{p_1-1}(\Xi_1') \otimes \hat{S}_{p_2}^{Per}(\Xi_2)\Big),\\
    \hat{X}^2_h &\coloneqq \hat{S}_{p_1-1}(\Xi_1') \otimes \hat{S}_{p_2-1}^{Per}(\Xi_2'),
\end{align}
and considering the projectors:
\begin{align}
\label{eqn:4.13}
    \hat{\Pi}^1 &= (\hat{\pi}_{p_1} \otimes \hat{\pi}_{p_2-1}^{c,Per}) \times( \hat{\pi}_{p_1-1}^c\otimes \hat{\pi}_{p_2}^{Per}),\\
    \hat{\Pi}^2 &= \hat{\pi}_{p_1-1}^c \otimes \hat{\pi}_{p_2-1}^{c,Per}.
\end{align}
\end{rmk}

\subsubsection{Semi-discretization in space}  \label{sec:semi-discrete_space}
We propose to take as a linear operator $\mathcal{T}$ in \eqref{eqn:2.7} and \eqref{eqn:2.8} the generic tensor product 
commutative projection $\Pi^2$. The semi-discrete problem in space reads, find $\mathbf{v}_h \in X^1_h$ and $\phi_h \in X^2_h$,
such that the following equations hold: 
\begin{subequations}
    \label{eqn:semi_disc_in_spazio}
    \begin{align}
        \label{eqn:semi_disc_in_spazio_a}
        \int_\Omega \left(\mathbf{v}_h\right)_t\cdot \mathbf{w}_h\ \mathrm{d}\mathbf{x} &=
         -\int_\Omega \Pi^2(\text{div}(c\mathbf{w}_h))\  \phi_h\ \mathrm{d}\mathbf{x}, \quad &\forall \mathbf{w}_h \in {X}^1_h, \\  
        \label{eqn:semi_disc_in_spazio_b}
        \left(\phi_h\right)_t &= \Pi^2(\text{div}(c\mathbf{v}_h)).            
    \end{align}
\end{subequations}
In next section we will exploit a particularization of such discretization, based on quasi-interpolant projections. 

\subsection{Quasi interpolant projections}\label{sec:sub_quasi_interpolant_projections}
\pgfplotsset{every axis/.append style={thick},compat=1.5}

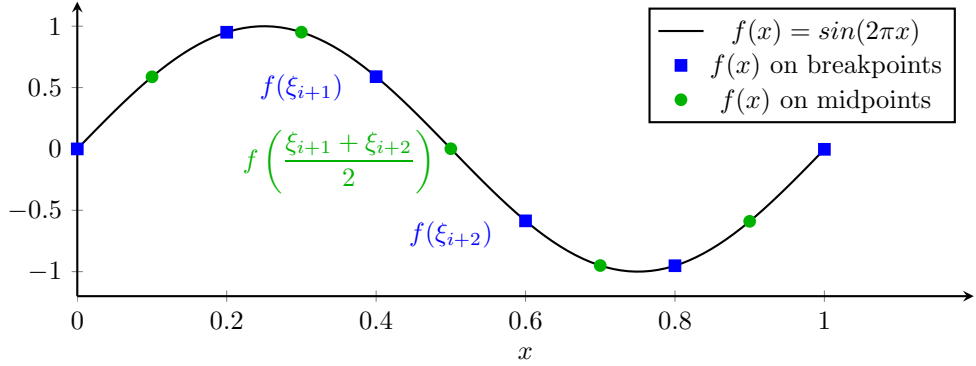
\begin{figure}
    \centering
    \begin{tikzpicture}
        \begin{axis}[
            axis lines=left,
            xlabel = \(x\),
            width=13.5cm,
            height=5.5cm,
            xmin=0.0,xmax=1.2,
            ymin=-1.2,ymax=1.2,
            ytick={-1,-0.5,0,0.5,1.0},
            xtick={0,0.2,0.4,0.6,0.8,1.0},
           legend entries = {$f(x)=sin(2\pi x)$,$f(x)$ on breakpoints ,$f(x)$ on midpoints}
        ]
        \addplot [domain=0:1, samples=100, color=black,]
            {sin(deg( 2*3.14*x)) };
        \addplot [only marks, domain=0:1, samples=6, color=blue, mark = square*] 
            {sin(deg( 2*3.14*x))};
        \addplot [only marks, domain=0.1:0.9, samples=5, color=green!70!black, mark = *]
            {sin(deg(2*3.14*x))};
        \node[blue] at (30,170) {$f(\xi_{i+1})$};
        \node[green!70!black] at (35,110) {$f\left(\dfrac{\xi_{i+1}+\xi_{i+2}}{2}\right)$};
        \node[blue] at (50,50) {$f(\xi_{i+2})$};
        \end{axis}
    \end{tikzpicture}
    \caption{Example of pointwise evaluation of $f(x) = sin(2\pi x)$ over breakpoints and midpoints 
    for projection with $\hat{\pi}_{2,\Xi}$ where $\Xi = \{0,0,0,0.2,0.4,0.6,0.8,1,1,1 \}$. 
    The three highlighted points are used for the computation of 
    $\lambda_{i,2}(f)$, for $i = 4$, corresponding to the B-spline with support $[0.2 , 0.8]$. }
    \label{fig:new_plot_1}
\end{figure}
Here we focus on a particular projection whose construction is given in detail in \cite{lee2000some}. We 
refer to such projection using the notation $\llm[p]$. We deal with open knot vectors with non-repeating 
internal knots, i.e. $\Xi = \{\xi_1, \dots, \xi_{n+p+1}\}$ such that the vector 
$\mathbf{Z} = \{\zeta_1, \dots, \zeta_{z}\}$ of the breakpoints has multiplicities 
$m_1 = m_z = p+1$ and $m_j = 1$ for $j = 2,\dots,z-1$. 
In order to find an approximation $\llm[p](f)$ of $f:[0,1]\to \mathbb{R}$, we need to define a dual 
basis $\lambda_{i,p}$ and compute $\lambda_{i,p}(f)$ for $i=1,\dots,n$, and the idea is the following:
\begin{enumerate}
    \item Choose $\mathcal{I} = [\xi_{\mu},\xi_{\nu}]\subset[\xi_{p+1},\xi_n]$ such that 
            $\mathcal{I} \cap [\xi_{i},\xi_{i+p+1}] \neq \emptyset$. Notice that $[\xi_{i},\xi_{i+p+1}]$ 
            is the support of $\hat{B}_{i,p}$, and denote by $f_{\mathcal{I}}$ the restriction of $f$ to $\mathcal{I}$.
    \item Choose a local approximation method $\mathcal{P}_{loc}$ for $f_{\mathcal{I}}$, such that it reproduces 
            splines of degree up to $p$. This approximation is written analogously to \eqref{eqn:3.2} in this way: 
            $\mathcal{P}_{loc}(f_{\mathcal{I}}) = \sum_{j = \mu-p}^{\nu-1} c_j \hat{B}_{j,p}$, and we have 
            $\mu-p\leq i \leq \nu-1$ since the support of $\hat{B}_{i,p}$ intersects $\mathcal{I}$.
    \item Finally set $\lambda_{i,p}(f) = c_i$.
\end{enumerate}
The idea is to use polynomial interpolation as local approximation method. We exploit the construction for spline 
spaces of degree $2$ and $3$ given in \cite{lee2000some}.
Let us fix now $p = 2$, for a given index $i = 1,\dots,n$, we have 
$[\xi_{i+1},\xi_{i+2}]\subset [\xi_{i}, \xi_{i+3}] = supp(\hat{B}_{i,2}) $, but 
$\xi_{i+1}<\xi_{i+2}$ when $i\neq 1,n$. This means that for $i = 2, \dots, n-1$
we can choose $\mathcal{I} = [\xi_{i+1}, \xi_{i+2}]$, and perform three points interpolation on 
$\xi_{i+1}< (\xi_{i+1}+\xi_{i+2})/2<\xi_{i+2}$. Notice that we are interpolating over two 
consecutive breakpoints and their midpoint in a uniform partition. Due to this particular 
choice of the interpolation points, we can explicitly write the coefficients of $\llm[2]$, that are:
\begin{equation}
    \label{eqn:degree2}
        \lambda_{i,2}(f) = -\frac{1}{2}f(\xi_{i+1}) + 2 f\Big(\frac{\xi_{i+1} + \xi_{i+2}}{2}\Big) -\frac{1}{2}f(\xi_{i+2}),\quad \text{for } i = 2,\dots,n-1.
\end{equation}
We remark this expression is valid whenever $\xi_{i+1} < \xi_{i+2}$,
which is not the case for $i = 1$ and $i = n$ in the case 
of an open knot vector, as used for Dirichlet boundary 
conditions. In this situation we want the operator to be 
interpolant at the boundary and hence we fix 
$\lambda_{1,2} = f(0) $ and $\lambda_{n,2} = f(1)$. As 
regards closed knot vectors for periodic boundary 
conditions, the computation is easier, since we can use 
\eqref{eqn:degree2} for all indices $i = 1,\dots,n$. Notice 
that we need the pointwise evaluation of the function $f$ 
we want to approximate, among the breakpoints and their 
midpoints, as in the example given in Figure \ref{fig:new_plot_1}.

We recall now the explicit formulae in the case $p=3$, that is the quasi-interpolant $\llm[3]$, and refer for details to \cite{lee2000some}. 
Here, as local approximation method, we use five points interpolation over the knots and midpoints of 
$\mathcal{I} = [\xi_{i+1}, \xi_{i+3}]$. Instead of \eqref{eqn:degree2}, we can write now
\begin{equation}
\label{eqn:degree3}
    \lambda_{i,3}(f) = \frac{1}{6}f(\xi_{i+1}) -\frac{4}{3}f\Big(\frac{\xi_{i+1} + \xi_{i+2}}{2}\Big) +\frac{10}{3}f(\xi_{i+2}) -\frac{4}{3}f\Big(\frac{\xi_{i+2} + \xi_{i+3}}{2}\Big) +\frac{1}{6}f(\eta_{i+3}).
\end{equation}
Again for open knot vectors, hence Dirichlet boundary conditions, we have the particular cases for 
$i = 1,2,n-1,n,$ that we recall here:
\begin{equation*}
    \begin{aligned}
    \lambda_{1,3}(f)  &= f(0),\\
    \lambda_{2,3}(f)  &= -\frac{5}{18}f(\xi_{4}) +\frac{20}{9}f\Big(\frac{\xi_{4} + \xi_{5}}{2}\Big) -\frac{4}{3}f(\xi_{5}) +\frac{4}{9}f\Big(\frac{\xi_{5} + \xi_{6}}{2}\Big) -\frac{1}{18}f(\xi_{6}),\\
    \lambda_{n-1,3}(f)&= -\frac{1}{18}f(\xi_{n-1}) +\frac{4}{9}f\Big(\frac{\xi_{n-1} + \xi_{n}}{2}\Big) -\frac{4}{3}f(\xi_{n}) +\\
                      &\quad +\frac{20}{9}f\Big(\frac{\xi_{n} + \xi_{n+1}}{2}\Big) -\frac{5}{18}f(\xi_{n+1}),\\
    \lambda_{n,3}(f)  &= f(1),
\end{aligned}
\end{equation*}
From the computational point of view, we remember this technique 
requires evaluation of $f$ over the breakpoints and midpoints. 
Notice that for projections in periodic spline spaces, it is sufficient
to use \eqref{eqn:degree2} or \eqref{eqn:degree3}.
Finally, for approximation estimates by quasi-interpolant projections based on 
point evaluation of functions, we refer to \cite[Section 5]{lyche1975local}.

\subsubsection{Commutative quasi-interpolant projections}
\label{sec:sub_sub_commutative_projections} 
In order to compute a projection $\llm[p]^c$ that commutes with $\llm[p]$, we 
follow the construction in \eqref{eqn:3.4}. Since we have to apply $\llm[p]$ to the integral function $F$, we need its 
evaluation over all the breakpoints and midpoints. For this reason, 
instead of using Gaussian quadrature, we prefer to use Cavalieri-Simpson composite 
quadrature formulae. We want to consider both breakpoints and 
midpoints in a unified notation. Hence we introduce 
$\boldsymbol{\eta} = \{\eta_1, \dots, \eta_{2z-1}\} $ such that:
\begin{equation*}
    \begin{cases}
        \eta_{2i-1} = \zeta_{i}, \quad &\text{for } i = 1, \dots, z,\\
        \eta_{2i} = \frac{1}{2}(\zeta_{i}+\zeta_{i+1}), \quad &\text{for } i = 1,\dots, z-1.
    \end{cases} 
\end{equation*} 

With this notation we have $F(\eta_1)=0$ and for $i = 2,\dots,2z-1$:
\begin{align*}
    F(\eta_i) &= \int_{0}^{\eta_i} f(s) \mathrm{d}s= \sum_{j = 1}^{i-1} \int_{\eta_j}^{\eta_{j+1}} f(s)\mathrm{d}s \\
              &\approx \sum_{j = 1}^{i-1} \frac{h}{6} \left( f(\eta_j) + 4f\Big(\frac{\eta_j + \eta_{j+1}}{2}\Big) + f(\eta_{j+1})\right),
\end{align*}
where $h = \eta_{j+1}-\eta_{j}$. Notice that this quadrature 
formula is exact for polynomials up to degree $3$, but it requires evaluation 
of the integrating function over further midpoints, i.e. 
$(\eta_j + \eta_{j+1})/2$, as it is shown in Figure \ref{fig:new_plot_2}. 
\begin{figure}
    \centering
    \begin{tikzpicture}
        \begin{axis}[
            axis lines=left,
            xlabel = \(x\),
            width=13.5cm,
            height=6cm,
            xmin=0.0,xmax=1.2,
            ymin=-1.2,ymax=1.2,
            ytick={-1,-0.5,0,0.5,1.0},
            xtick={0,0.2,0.4,0.6,0.8,1.0},
            legend entries = {$f(x)=sin(2\pi x)$,$f(x)$ on breakpoints, $f(x)$ on midpoints, on further midpoints}
            ]
        \addplot [domain=0:1, samples=100, color=black,]
            {sin(deg( 2*3.14*x)) };
        \addplot [only marks, domain=0:1, samples=6, color=blue, mark = square*] 
            {sin(deg( 2*3.14*x))};
        \addplot [only marks, domain=0.1:0.9, samples=5, color=green!60!black, mark = *]
            {sin(deg(2*3.14*x))};
        \addplot [only marks, domain=0.05:0.95, samples=10, color=red, mark = star]
            {sin(deg(2*3.14*x))};
        \node[blue] at (35,170) {$f(\eta_{j})$};
        \node[green!60!black] at (43,110) {$f(\eta_{j+1})$};
        \node[blue] at (53,50) {$f(\eta_{j+2})$};
        \node[red] at (59,160) {$f\left(\dfrac{\eta_{j} + \eta_{j+1}}{2}\right)$};
        \node[red] at (71,100) {$f\left(\dfrac{\eta_{j+1} + \eta_{j+2}}{2}\right)$};
        \end{axis}
    \end{tikzpicture}
    \caption{Example of evaluation of $f(x) = sin(2\pi x)$ in further midpoints, in order to project with $\llm[p]^c$}
    \label{fig:new_plot_2}
\end{figure}
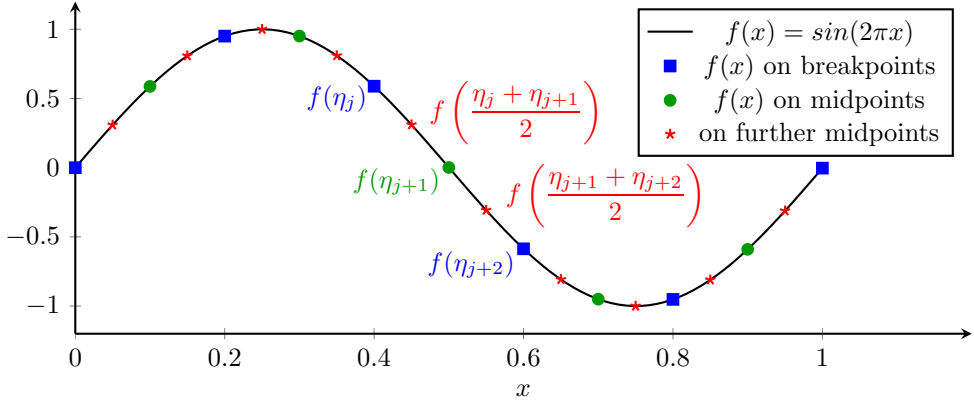

\subsubsection{Multivariate quasi-interpolant projections}
\label{sec:sub_sub_multivariate_commutative_projections} 
Finally we define the multivariate quasi-interpolant projections that commute with the 
divergence operator in the parametric domain, as 
\begin{equation}\label{eqn:multivariate_llm_projections}
    \begin{aligned}
        \LLM[1] &= (\llm[p_1] \otimes \llm[p_2-1]^c) \times( \llm[p_1-1]^c\otimes \llm[p_2]),\\
        \LLM[2] &= \llm[p_1-1]^c \otimes \llm[p_2-1]^c.    
    \end{aligned}
\end{equation}
The projections in the physical domain, namely $\pLLM[1],\pLLM[2]$, are uniquely determined by  
\eqref{eqn:4.6new}. We remark that projections $\LLM[1],\LLM[2],\pLLM[1],\pLLM[2]$ satisfy \cref{lemma:commutativity of projections}.
An analogous constructions holds for projections in periodic spline spaces, or in 
spline spaces with mixed Dirichlet and periodic conditions, as in \cref{rem:mixed_projections}. 


\subsection{Conservative time discretization}
\label{sec:sub_conservative_semi_discretization_in_time}
Regarding time discretization, the choice of an energy-preserving method is mandatory since we are interested in preserving the total 
energy of the system. We detail the discretization with Crank-Nicolson method, which is second order and energy conservative, but other conservative methods could be chosen.

Let us fix $\tau = (t_0,\dots,t_{N_T})$, a partition of the interval $[0,T]$, such that $t_0=0$, $t_{N_T}=T$ and $t_n < t_{n+1}$.
For a simplified notation we assume the partition to be uniform, that is $t_{n+1}-t_n= k$, with a fixed real number $k>0$, 
for $ n= 0, \dots , N_T-1$. Let us denote by $\mathbf{v}_h^n$ and $\phi_h^n$ respectively the value of $\mathbf{v}_h $ and $\phi_h$ 
at the time instant $t_n$. Applying Crank-Nicolson method to equations \eqref{eqn:semi_disc_in_spazio},
the fully discrete problem is now, for $ n = 0, \dots ,N_T-1 $ find $\mathbf{v}_h^{n+1} \in X^1_h$ and $\phi_h^{n+1} \in X^2_h$, 
such that
\begin{subequations}
\label{eqn:5.6}
\begin{align}
    \int_{\Omega} \frac{\mathbf{v}_h^{n+1}-\mathbf{v}_h^n}{k}\cdot\mathbf{w}_h \mathrm{d}\mathbf{x} &=
     -\frac{1}{2}\int_{\Omega} \Pi^{2} (\text{div}(c\mathbf{w}_h))(\phi_h^{n+1} + \phi_h^n)\mathrm{d}\mathbf{x},
      \quad \forall \mathbf{w}_h \in X^1_h,\label{eqn:4.11a}\\
    \frac{\phi_h^{n+1} - \phi_h^n}{k} &= + \frac{1}{2}\big( \Pi^{2} (\text{div}(c\mathbf{v}_h^{n+1}))+ \Pi^{2} (\text{div}( c\mathbf{v}_h^n))\big) 
    \label{eqn:4.11b}.
\end{align}
\end{subequations}
Notice that in equation \eqref{eqn:4.11b}, the left hand side $\phi_h^{n+1}$ is written as a linear combination of terms that depend 
on the time instant $t_n$, except for $\mathbf{v}_h^{n+1}$. Moreover, in view of the definition of the operator $\Pi^2$, 
\eqref{eqn:4.11b} can be seen as a collocation of the equation \eqref{eqn:2.8}. To solve the system, 
we replace the expression of $\phi_h^{n+1}$ from the second equation into the first one, and bring to the left all the terms in which the unknown 
$\mathbf{v}_h^{n+1}$ appears, that is
\begin{multline}
\label{eqn:5.8}
    \int_{\Omega}\mathbf{v}_h^{n+1}\cdot \mathbf{w}_h\mathrm{d}\mathbf{x} + 
     \frac{k^2}{4}\int_{\Omega} \Pi^{2} (\text{div} (c\mathbf{v}_h^{n+1})) \Pi^{2}(\text{div}(c\mathbf{w}_h))\mathrm{d}\mathbf{x}  = \\
    =\int_{\Omega} \mathbf{v}_h^n\cdot \mathbf{w}_h\mathrm{d}\mathbf{x} - 
     \frac{k^2}{4}\int_{\Omega} \Pi^{2}(\text{div}(c\mathbf{v}_h^n))\Pi^{2}(\text{div}(c\mathbf{w}_h))\mathrm{d}\mathbf{x} - 
      k\int_{\Omega}\phi_h^n\Pi^{2}(\text{div}(c\mathbf{w}_h))\mathrm{d}\mathbf{x},  
\end{multline}
which must hold for all $\mathbf{w}_h \in X^1_h$. We solve this equation to compute the unknown $\mathbf{v}_h^{n+1}$, which is then used to update the solution $\phi_h^{n+1}$ as indicated by the second equation of \eqref{eqn:5.6}.

In our numerical tests, we will to compare the proposed method with the standard Galerkin formulation, that we present for completeness, and which is given by: find $\mathbf{v}_h \in X^1_h$ and $\phi_h \in X^2_h$, such that 
\begin{subequations}
\label{eqn:5.7}
\begin{align}
    \int_{\Omega} \frac{\mathbf{v}_h^{n+1}-\mathbf{v}_h^n}{k}\cdot\mathbf{w}_h \mathrm{d}\mathbf{x} &= 
     -\frac{1}{2}\int_{\Omega} \text{div}(c\mathbf{w}_h)(\phi_h^{n+1} + \phi_h^n)\mathrm{d}\mathbf{x}, 
      \quad \forall \mathbf{w}_h \in X^1_h, \label{eqn:5.7a}\\
    \int_{\Omega} \frac{\phi_h^{n+1}-\phi_h^n}{k}\psi_h \mathrm{d}\mathbf{x} &= 
     + \frac{1}{2}\int_{\Omega}\big( \text{div}(c\mathbf{v}_h^{n+1}) + \text{div}( c\mathbf{v}_h^n)\big)\psi_h\mathrm{d}\mathbf{x},
      \quad \forall \psi_h \in X^2_h, \label{eqn:5.7b}
\end{align}
\end{subequations}
\begin{rmk}\label{Remark:energy_conservation}
Notice that both equations \eqref{eqn:5.6} and \eqref{eqn:5.7} are energy conservative schemes by construction. On the other hand, 
\eqref{eqn:5.7} can not be reduced to a single equation. Thus, the method \eqref{eqn:5.6} is intrinsically cheaper than the method 
\eqref{eqn:5.7} and, at our knowledge, it is the first conservative method that does not require computation of both primal and dual 
variables at each time step. { REWRITE THIS SENTENCE? DISCUSS}
\end{rmk}


\section{Error convergence analysis}
\label{sec:error_analysis}
In this section we analyze the convergence of the proposed method. In Section~\ref{sec:sub_convergence_analysis_with_adjoint_operator}
we first present a convergence study in an abstract setting, 
for a generic family of projections as in Section~\ref{sec:semi-discrete_space}, 
assuming stability and approximation properties both for the projectors and 
for their adjoint operators. Unfortunately, the adjoints of 
the quasi-interpolants of Section~\ref{sec:sub_quasi_interpolant_projections} 
do not satisfy these assumptions. For this reason, we present in Section~\ref{sec:sub_error_analysis_with_relaxed_assumptions}
the analysis under some relaxed assumptions, for which only 
linear convergence could be proved.

It will be useful for the analysis to write the equations of the problem in weak form. Let us denote by 
$(\mathbf{v},\phi)$ the classical solution of the problem \cref{eqn:2.2}, while $(\mathbf{v}_h,\phi_h)$ is the solution of the 
semidiscrete problem in space \eqref{eqn:semi_disc_in_spazio}, which satisfies also the following variational equations:
\begin{subequations}
\label{eqn:5.1}
\begin{align}
    \label{eqn:5.1a}
    \int_\Omega \left(\mathbf{v}_h\right)_t\cdot \mathbf{w}_h\ \mathrm{d}\mathbf{x} + \int_\Omega \Pi^2(\text{div}(c\mathbf{w}_h))\  \phi_h\ \mathrm{d}\mathbf{x} &=0, \quad &\forall \mathbf{w}_h \in {X}^1_h, \\  
    \label{eqn:5.1b}
    \int_\Omega \left(\phi_h\right)_t \psi_h\ \mathrm{d}\mathbf{x}  -\int_\Omega \Pi^2(\text{div}(c\mathbf{v}_h))\ \psi_h\ \mathrm{d}\mathbf{x} &= 0, \quad &\forall \psi_h \in {X}^2_h.
\end{align}
\end{subequations}
We introduce the energy norm $\|\mathbf{v},\phi\|_{E}^2 \coloneqq \|\mathbf{v}\|_{\mathbf{L}^2(\Omega)}^2+\|\phi\|_{L^2(\Omega)}^2 $, the Sobolev norm
$\|\mathbf{v},\phi\|_{\mathcal{H}^s}^2 \coloneqq  \|\mathbf{v}\|_{\mathbf{H}^s(\text{div};\Omega)}^{2} + \|\phi\|_{H^s(\Omega)}^{2} $, for a positive integer $s\geq0$, 
and the following norms in the space time domain $\Omega \times [0,T]$
\begin{align*}
    \|\mathbf{v}, \phi \|_{\infty,E} &\coloneqq \sup_{t\in [0,T]} \|\mathbf{v},\phi\|_{E},\\
    \|\mathbf{v}, \phi \|_{W^{1,\infty},\mathcal{H}^s} &\coloneqq \max \{ \sup_{t\in [0,T]} \|\mathbf{v},\phi\|_{\mathcal{H}^s}, \sup_{t\in [0,T]} \|\mathbf{v}_t,\phi_t\|_{\mathcal{H}^s}\}.
\end{align*}
In order to alleviate notation, we indicate by $\|\phi,\psi\|_{E}^2 = \| \phi \|_{L^2(\Omega)}^2 + \|\psi\|_{L^2(\Omega)}^2 $ the equivalent of the energy norm when $\phi,\psi \in L^2({\Omega})$ are both scalar fields.  
We are interested in bounding 
\begin{equation}
\label{eqn:5.2}
    \|\mathbf{v}-\mathbf{v}_h,\phi-\phi_h\|_{\infty,E} = \sup_{t\in [0,T]} \|\mathbf{v}-\mathbf{v}_h,\phi-\phi_h\|_{E},
\end{equation}
and we will follow the ideas used in \cite{arnold2014mixed} for elasticity problems in mixed form with weak symmetry.

\subsection{Convergence analysis based on the adjoint projections}
\label{sec:sub_convergence_analysis_with_adjoint_operator}

{
Let us introduce the adjoint projection  $\Pi^{2,*} \in \textit{End}(X^2)$, which is defined by the following fundamental relation
\begin{equation}\label{eqn:def_adjoint_projection}
    \prescript{}{X^2}\langle \Pi^2(\phi), \psi \rangle_{X^{2,*}} = \prescript{}{X^2}\langle \phi , \Pi^{2,*} (\psi)\rangle_{X^{2,*}}, \quad
    \forall \ \phi \in X^2,\ \psi  \in X^{2,*}, 
\end{equation}
where $X^{2,*} $ is the dual space of $X^2$, and $\langle\cdot,\cdot\rangle$ denotes the duality pairing.
We identify $X^2 = L^2(\Omega)$ with its dual space, as it is usually done, and in particular \eqref{eqn:def_adjoint_projection} can be expressed in terms of 
the $L^2(\Omega)$ scalar product. Throughout this section, we require that $\Pi^2$ and $\Pi^{2,*}$ are $L^2$-stable, and $\Pi^1,\Pi^2$ and $\Pi^{2*}$ have good approximation properties, summarizing our request in the following assumption.  
\begin{ass}
\label{assumption_3}
The projections $\Pi^2$ and $\Pi^{2,*}$ are $L^2$-stable, that is 
\begin{align*}
    \|\Pi^2(\phi)\|_{L^2(\Omega)}         \leq \|\phi\|_{L^2(\Omega)},
    \quad \text{and} \quad
    \|\Pi^{2,*}(\psi)\|_{L^2(\Omega)}     \leq \|\psi\|_{L^2(\Omega)},
    \quad
    \forall \phi,\, \psi \in L^2(\Omega).
\end{align*}
Moreover, given $\mathbf{w} \in \mathbf{H}^m(\mathrm{div};\Omega)$ and $\psi\in H^m(\Omega)$, for $0\leq l \leq m \leq p$, and $p = \min_{i=1}^d\{p_i\}$, there exists a constant $C$ such that it holds: 
\begin{subequations}\label{eqn:5.3}
\begin{align}
    \label{eqn:5.3.a}
    \|(Id - \Pi^1)\mathbf{w}\|_{\mathbf{H}^l(\mathrm{div};\Omega)}  &\leq C h^{m-l} \| \mathbf{w}\|_{\mathbf{H}^m(\mathrm{div};\Omega)},\\ 
    \label{eqn:5.3.b}
    \|(Id - \Pi^2)\psi\|_{{H}^l(\Omega)}  &\leq C h^{m-l} \| \psi \|_{{H}^m(\Omega)},\\ 
    \label{eqn:5.3.b*}
    \|(Id - \Pi^{2,*})\psi\|_{{H}^l(\Omega)} &\leq C h^{m-l} \| \psi \|_{{H}^m(\Omega)},
\end{align}
\end{subequations}
where $Id$ is the identity operator, and $h$ is the mesh size. 
\end{ass}}
{  
Projections that commute with the divergence operator and satisfy $L^2$-stability, together with \eqref{eqn:5.3.a} and \eqref{eqn:5.3.b}, are well known in literature,  see \cite[Remark 5.1]{buffa2011isogeometric}.
The usefulness of this working hypothesis is thus to ensure \eqref{eqn:5.3.b*}.
One option is to explore quasi-interpolant operators constructed with biorthogonal dual bases, the latter having the capability of reproducing polynomials, see i.e. \cite{oswald2001polynomial}.
}

Let us assume from now on that the classical solution is sufficiently 
regular, i.e. $\mathbf{v}\in \mathbf{H}^m(\text{div};\Omega)$ and 
$\phi\in H^m(\Omega)$, for $0\leq m \leq p$ and $p$ as in Assumption 
\ref{assumption_3}. We can start bounding \eqref{eqn:5.2} by 
using the triangular inequality, and separate the error in 
the following way: 
\begin{equation}
\label{eqn:5.3.5}
     \|\mathbf{v}-\mathbf{v}_h,\phi-\phi_h\|_{\infty,E} \leq  \|\mathbf{v}-\mathbf{v}_h^P,\phi-\phi_h^P\|_{\infty,E} +  \|\mathbf{v}_h^P-\mathbf{v}_h,\phi_h^P-\phi_h\|_{\infty,E},
\end{equation}
where $\mathbf{v}_h^P$ and $\phi_h^P$ are suitable approximations of the fields $\mathbf{v}$ and $\phi$.
{ Here we fix $\mathbf{v}_h^P \coloneqq \Pi^1(\mathbf{v})$ and $\phi_h^P \coloneqq \Pi^2(\phi)$, 
but other approximation choices that satisfy \eqref{eqn:5.3.a} and \eqref{eqn:5.3.b} can also be considered.}
We recall the following approximation result.
\begin{lem}\label{lemma5.1}
There exists a constant $C>0$ such that, for $0\leq m \leq p$ and $p$ as in Assumption 
\ref{assumption_3}, it holds
\begin{subequations}
\label{eqn:5.4}
\begin{align}
\label{eqn:5.4a}
\|{\mathbf{v}-\mathbf{v}_h^P,\phi-\phi_h^P}\|_{\infty,E} &\leq Ch^m \|{\mathbf{v},\phi}\|_{\infty,\mathcal{H}^m},\\ \label{eqn:5.4b}
\|{\phi-\phi_h^P,\ \textup{div}(c\mathbf{v}-c\mathbf{v}_h^P)}\|_{\infty,E} &\leq Ch^{m} \|{\mathbf{v},\phi}\|_{\infty,\mathcal{H}^m},\\ \label{eqn:5.4c}
\|{(\mathbf{v}-\mathbf{v}_h^P)_t,(\phi-\phi_h^P)_t}\|_{\infty,E} &\leq Ch^m \|{\mathbf{v}_t,\phi_t}\|_{\infty,\mathcal{H}^m}.
\end{align}
\end{subequations}
\end{lem}
\begin{proof}
    In order to bound \eqref{eqn:5.4a} we argue like this:
    \begin{equation*}
        \begin{aligned}
            \|{\mathbf{v}-\mathbf{v}_h^P,\phi-\phi_h^P}\|_{\infty,E}^2 
            & \eqdef \sup_{t\in [0,T]} \left( \|\mathbf{v}-\mathbf{v}_h^P\|_{\mathbf{L}^2(\Omega)}^2 + \|\phi-\phi_h^P\|_{L^2(\Omega)}^2 \right)\\
            & \leq \sup_{t\in [0,T]}  \left( Ch^{2m} \|\mathbf{v}\|_{\mathbf{H}^m(\text{div};\Omega)}^{2} +  Ch^{2m}\|\phi\|_{H^m(\Omega)}^2\right)\\
            & = Ch^{2m} \|{\mathbf{v},\phi}\|_{\infty, \mathcal{H}^m}^2,               
        \end{aligned}
    \end{equation*}
    where we used only \eqref{eqn:5.3.a} and \eqref{eqn:5.3.b}, and taking the square roots we proved \eqref{eqn:5.4a}.
    The proof of \eqref{eqn:5.4c} is analogous. As regards \eqref{eqn:5.4b} we have that
    \begin{equation*}
            \|{\phi-\phi_h^P, \textup{div}(c\mathbf{v}-c\mathbf{v}_h^P)}\|_{\infty,E}^2 
            = \sup_{t\in [0,T]} \left( \|\phi-\phi_h^P\|_{L^2(\Omega)}^2 + \|\textup{div}(c\mathbf{v}-c\mathbf{v}_h^P)\|_{L^2(\Omega)}^2\right).
    \end{equation*}
    By applying the Leibnitz rule and Young's inequality, $(a+b)^2\leq 2a^2+2b^2 $ for each $a,b>0$, on the last term of the right hand side we have 
    \begin{equation*}
        \begin{aligned}
            \|\textup{div}(c\mathbf{v}-c\mathbf{v}_h^P)\|_{L^2(\Omega)}^2 &= 
            \left( \|c\ \textup{div}(\mathbf{v}-\mathbf{v}_h^P)\|_{L^2(\Omega)} + \|\grad{c} \cdot (\mathbf{v}-\mathbf{v}_h^P)\|_{L^2(\Omega)}\right)^2 \\
            &\leq 2 \|c\ \textup{div}(\mathbf{v}-\mathbf{v}_h^P)\|_{L^2(\Omega)}^2 + 2 \|\grad{c} \cdot (\mathbf{v}-\mathbf{v}_h^P)\|_{L^2(\Omega)}^2\\
            & \leq  C \|\mathbf{v}-\mathbf{v}_h^P\|_{\mathbf{H}^0(\text{div};\Omega)}^{2},           
        \end{aligned}
    \end{equation*}
    because $c$ is smooth and uniformly bounded. Now again we can apply \eqref{eqn:5.3.a} and \eqref{eqn:5.3.b},
    which gives us
    \begin{equation*}
        \begin{aligned}
            \|{\phi-\phi_h^P, \textup{div}(c\mathbf{v}-c\mathbf{v}_h^P)}\|_{\infty,E}^2 
            & \leq C \sup_{t\in [0,T]}   \left( \|\phi-\phi_h^P \|_{L^2(\Omega)}^2 +  \|\mathbf{v}-\mathbf{v}_h^P\|_{\mathbf{H}^0(\text{div};\Omega)}^{2} \right)\\
            & \leq C \sup_{t\in [0,T]}  \left( h^{2m}  \|\phi \|_{H^m(\Omega)}^2 + h^{2m} \|\mathbf{v}\|_{\mathbf{H}^m(\text{div};\Omega)}^{2} \right)\\
            & \leq C h^{2m} \|{\mathbf{v},\phi}\|_{\infty, \mathcal{H}^m}^2.
        \end{aligned}
    \end{equation*}
    and again taking the square roots we end the proof of this lemma. 
\end{proof}
Next we state the main result of this paper.
\begin{thm}\label{Theorem5.2}
Under Assumption \ref{assumption_3}, together with the commutativity of the projectors, 
$\Pi^2 \textup{div} = \textup{div} \Pi^1$, given $\mathbf{v} \in \mathbf{H}^m(\mathrm{div};\Omega)$ and $\phi \in H^m(\Omega)$, 
it holds that
\begin{equation}
\label{tesi}
    \|\mathbf{v}-\mathbf{v}_h,\phi-\phi_h\|_{\infty,E} \leq C h^{m-1} \|\mathbf{v},\phi\|_{W^{1,\infty},\mathcal{H}^m}.
\end{equation}
\end{thm}
\begin{proof}
Let us consider both the weak formulation \eqref{eqn:2.3and2.4} 
and the semi-discrete problem \eqref{eqn:5.1}. With the 
notation $(\cdot,\cdot)$ indicating the $L^2$ scalar product 
over the physical domain $\Omega$ we can write:
\begin{align*}
    \left((\mathbf{v}_h)_t,\mathbf{w}_h\right) + \left(\Pi^2\text{div}(c\mathbf{w}_h),\phi_h\right) &= \left(\mathbf{v}_t,\mathbf{w}_h\right) + \left(\text{div}(c\mathbf{w}_h),\phi\right),\\
    \left((\phi_h)_t,\psi_h\right) - \left(\Pi^2\text{div}(c\mathbf{v}_h),\psi_h\right) &= \left(\phi_t,\psi_h\right) - \left(\text{div}(c\mathbf{v}),\psi_h\right),    
\end{align*}
for all $\mathbf{w}_h \in X^1_h$ and $\psi_h \in X^2_h$. 
We can respectively subtract from the previous equations the following quantities:
$$
\left((\mathbf{v}_h^P)_t,\mathbf{w}_h\right) + \left(\Pi^2\textup{div}(c\mathbf{w}_h),\phi_h^P\right), \quad \text{and}\quad \left((\phi_h^P)_t,\psi_h\right) - \left(\Pi^2\textup{div}(c\mathbf{v}_h^P),\psi_h\right),
$$
and rearrange the two equations such that it holds that
\begin{multline}\label{first:equation:1}
    \left((\mathbf{v}_h-\mathbf{v}_h^P)_t,\mathbf{w}_h\right) + \left(\Pi^2\textup{div}(c\mathbf{w}_h),(\phi_h-\phi_h^P)\right) = \\
                   = \left((\mathbf{v}-\mathbf{v}_h^P)_t,\mathbf{w}_h\right) + \left(\Pi^2\textup{div}(c\mathbf{w}_h),(\phi-\phi_h^P)\right)+ \left((Id-\Pi^2)\textup{div}(c\mathbf{w}_h),\phi\right),    
\end{multline}
and 
\begin{multline}\label{second:equation:2}
        \left((\phi_h-\phi_h^P)_t,\psi_h\right) - \left(\Pi^2\textup{div}(c\mathbf{v}_h - c\mathbf{v}_h^P),\psi_h\right) = \\
                   = \left((\phi-\phi_h^P)_t,\psi_h\right) - \left(\Pi^2\textup{div}(c\mathbf{v} - c\mathbf{v}_h^P),\psi_h\right) - \left((Id-\Pi^2)\textup{div}(c\mathbf{v}),\psi_h\right).
\end{multline}
By making the choice $\mathbf{w}_h = \mathbf{v}_h - \mathbf{v}_h^P$ 
and $\psi_h = \phi_h - \phi_h^P$, and adding together the 
equations we get
\begin{align} \label{eq:same_steps}
\begin{split}
    \dfrac{1}{2}\pdv{}{t}\|\mathbf{v}_h &- \mathbf{v}_h^P, \phi_h - \phi_h^P\|_{E}^2 = \\ 
    &= \left((\mathbf{v}-\mathbf{v}_h^P)_t,\mathbf{v}_h-\mathbf{v}_h^P\right) + \left((\phi-\phi_h^P)_t,\phi_h-\phi_h^P\right)+\\ 
    &+\left(\Pi^2\textup{div}(c\mathbf{v}_h-c\mathbf{v}_h^P),\phi-\phi_h^P\right) -\left(\Pi^2\textup{div}(c\mathbf{v} - c\mathbf{v}_h^P),\phi_h-\phi_h^P\right)+\\
    &+\left((Id-\Pi^2)\textup{div}(c\mathbf{v}_h-c\mathbf{v}_h^P),\phi\right) - \left((Id-\Pi^2)\textup{div}(c\mathbf{v}),\phi_h-\phi_h^P\right).
\end{split}
\end{align}
Here, by construction $\Pi^2 \textup{div} = \textup{div}\, \Pi^1$. 
Combining this with the definition of the formal adjoint 
operator, we can rewrite the last two terms on the right hand side as
\begin{align} \label{eq:diff_step}
\begin{split}
    \left((Id-\Pi^2)\textup{div}(c\mathbf{v}_h-c\mathbf{v}_h^P),\phi\right) &= \left(\textup{div}(c\mathbf{v}_h^P-c\mathbf{v}_h),(Id-\Pi^{2,*})\phi\right),\\
    \left((Id-\Pi^2)\textup{div}(c\mathbf{v}),\phi_h-\phi_h^P\right) &= \left(\textup{div}(Id-\Pi^1)(c\mathbf{v}),\phi_h-\phi_h^P\right).
\end{split}
\end{align}
Now on the right hand side we can apply the Cauchy-Schwartz inequality, which gives  
\begin{align*}
    \begin{split}
        \dfrac{1}{2}\pdv{}{t}\|\mathbf{v}_h - \mathbf{v}_h^P, \phi_h - \phi_h^P\|_{E}^2 \leq \|(\mathbf{v}-\mathbf{v}_h^P)_t\|_{\mathbf{L}^2(\Omega)} &\|\mathbf{v}_h-\mathbf{v}_h^P\|_{\mathbf{L}^2(\Omega)}  + \\
        +\| (\phi-\phi_h^P)_t\|_{L^2(\Omega)}                                  &\|\phi_h-\phi_h^P\|_{L^2(\Omega)} +\\ 
        +\| (\phi-\phi_h^P)  \|_{L^2(\Omega)}                                  &\|  \Pi^2 \textup{div}\left(c\mathbf{v}_h^P-c\mathbf{v}_h\right) \|_{L^2(\Omega)} +\\ 
        +\| \Pi^2\textup{div}(c\mathbf{v} - c\mathbf{v}_h^P) \|_{L^2(\Omega)}  &\| \phi_h-\phi_h^P\|_{L^2(\Omega)} +\\
        +\| (Id-\Pi^{2,*})\phi \|_{L^2(\Omega)}                                &\| \textup{div}(c\mathbf{v}_h^P-c\mathbf{v}_h)\|_{L^2(\Omega)}  +\\
        +\| \textup{div}(Id-\Pi^1)(c\mathbf{v}) \|_{L^2(\Omega)}               &\| \phi_h-\phi_h^P\|_{L^2(\Omega)} .
    \end{split}
\end{align*}
By boundedness of the projections as in Assumption \ref{assumption_3} and using boundedness of the coefficient $c$, together with inverse inequalities, there exists a constant $C>0$ such that:
\begin{align*}
    \begin{split}
        \dfrac{1}{2}\pdv{}{t}\|\mathbf{v}_h - \mathbf{v}_h^P, \phi_h - \phi_h^P\|_{E}^2 \leq \|(\mathbf{v}-\mathbf{v}_h^P)_t\|_{\mathbf{L}^2(\Omega)} &\|\mathbf{v}_h-\mathbf{v}_h^P\|_{\mathbf{L}^2(\Omega)}  + \\
        +        \| (\phi-\phi_h^P)_t\|_{L^2(\Omega)}                                    &\|\phi_h-\phi_h^P\|_{L^2(\Omega)} +\\ 
        +Ch^{-1 }\| (\phi-\phi_h^P)  \|_{L^2(\Omega)}                                    &\|\mathbf{v}_h -\mathbf{v}_h^P\|_{\mathbf{L}^2(\Omega)} +\\ 
        +C       \| \textup{div}(c\mathbf{v} - c\mathbf{v}_h^P) \|_{L^2(\Omega)}         &\| \phi_h-\phi_h^P\|_{L^2(\Omega)} +\\
        +Ch^{-1 }\| (Id-\Pi^{2,*})\phi \|_{L^2(\Omega)}                                  &\| \mathbf{v}_h - \mathbf{v}_h^P\|_{\mathbf{L}^2(\Omega)}  +\\
        +        \| \textup{div}(Id-\Pi^1)(c\mathbf{v}) \|_{L^2(\Omega)}                 &\| \phi_h-\phi_h^P\|_{L^2(\Omega)} .
    \end{split}
\end{align*}
Again by Cauchy-Schwartz inequality, 
applied to ordered couples of terms in the right hand side, we have: 
\begin{align*} 
    \dfrac{1}{2}\pdv{}{t}\|\mathbf{v}_h - \mathbf{v}_h^P,\ \phi_h - \phi_h^P\|_{E}^2 \leq \|(\mathbf{v}-\mathbf{v}_h^P)_t,\ (\phi - \phi_h^P)_t\|_{E} \|\mathbf{v}_h - \mathbf{v}_h^P,\ \phi_h - \phi_h^P\|_{E} &+\\ 
    +Ch^{-1 }\|(\phi-\phi_h^P),\ \textup{div}(c\mathbf{v} - c\mathbf{v}_h^P) \|_{E} \|\mathbf{v}_h - \mathbf{v}_h^P,\ \phi_h - \phi_h^P\|_{E} &+ \\
    +Ch^{-1 }\|(Id-\Pi^{2,*})\phi,\ \textup{div}(Id-\Pi^1)(c\mathbf{v})\|_{E} \|\mathbf{v}_h - \mathbf{v}_h^P,\ \phi_h - \phi_h^P\|_{E}&.
\end{align*}
Dividing both sides by $\|\mathbf{v}_h - \mathbf{v}_h^P,\ \phi_h - \phi_h^P\|_{E}$ and 
integrating in time in $[0,t]$, we obtain
\begin{align*}
    \|\mathbf{v}_h - \mathbf{v}_h^P,\ \phi_h - \phi_h^P\|_{E} &\leq \|\mathbf{v}_h(\mathbf{x},0) - \mathbf{v}_h^P(\mathbf{x},0),\ \phi_h(\mathbf{x},0) - \phi_h^P(\mathbf{x},0)\|_{E} + \\
    & +\int_0^t \|(\mathbf{v}-\mathbf{v}_h^P)_t,\ (\phi - \phi_h^P)_t\|_{E} \mathrm{d}\tau +\\
    & +\int_0^t Ch^{-1 }\|(\phi-\phi_h^P),\ \textup{div}(c\mathbf{v} - c\mathbf{v}_h^P) \|_{E} \mathrm{d}\tau +\\
    & +\int_0^t Ch^{-1 }\|(Id-\Pi^{2,*})\phi,\ \textup{div}(Id-\Pi^1)(c\mathbf{v})\|_{E} \mathrm{d}\tau.
\end{align*}
Specifying the initial condition for problem \eqref{eqn:5.1} to be 
$\left(\mathbf{v}_h^P(\mathbf{x},0), \phi_h^P(\mathbf{x},0)\right)$ 
it holds $ \|\mathbf{v}_h(\mathbf{x},0) - \mathbf{v}_h^P(\mathbf{x},0),\ \phi_h(\mathbf{x},0) - \phi_h^P(\mathbf{x},0)\|_{E}=0$. 
Since the previous estimate is valid for almost every $t \in[0,T]$ 
we can take the maximum, and there exists a new constant $C>0$ such that 
\begin{align*}
\begin{split}
    \|\mathbf{v}_h - \mathbf{v}_h^P,\ \phi_h - \phi_h^P\|_{\infty,E} \leq  
      C       &\|(\mathbf{v}-\mathbf{v}_h^P)_t,\ (\phi - \phi_h^P)_t\|_{\infty,E} +\\
    + Ch^{-1 }&\|(\phi-\phi_h^P),\ \textup{div}(c\mathbf{v} - c\mathbf{v}_h^P)  \|_{\infty,E} +\\
    + Ch^{-1} &\|(Id-\Pi^{2,*})\phi,\ \textup{div}(Id-\Pi^1)(c\mathbf{v})\|_{\infty,E}.
\end{split}
\end{align*}
Now the left hand side in \eqref{tesi} splits as in \eqref{eqn:5.3.5}, therefore we have 
\begin{align*}
    \begin{split}
        \|\mathbf{v} - \mathbf{v}_h,\ \phi - \phi_h\|_{\infty,E} \leq  \
                  &\|\mathbf{v}-\mathbf{v}_h^P,\ \phi - \phi_h^P\|_{\infty,E} +\\
        + C       &\|(\mathbf{v}-\mathbf{v}_h^P)_t,\ (\phi - \phi_h^P)_t\|_{\infty,E} +\\
        + Ch^{-1} &\|(\phi-\phi_h^P),\ \textup{div}(c\mathbf{v} - c\mathbf{v}_h^P)  \|_{\infty,E} +\\
        + Ch^{-1} &\|(Id-\Pi^{2,*})\phi,\ \textup{div}(Id-\Pi^1)(c\mathbf{v})\|_{\infty,E}.
        \end{split}
\end{align*}    
Finally, from Lemma \ref{lemma5.1} and Assumption \ref{assumption_3} we obtain
\begin{align*}
    \begin{split}
        \|\mathbf{v} - \mathbf{v}_h,\ \phi - \phi_h\|_{\infty,E} \leq C 
        &\Big(h^{m-1} \|\mathbf{v},\ \phi \|_{\infty,\mathcal{H}^m} + 
              h^{m} \|\mathbf{v}_t,\ \phi_t \|_{\infty,\mathcal{H}^m} \Big),
    \end{split}
\end{align*}    
which concludes the proof of this theorem.
\end{proof}

\begin{rmk}\label{remark:stima_ottimale}
    The suboptimality of the estimate comes from using inverse inequalities to control the 
    divergence of the error. Optimality can be recovered if we define suitable approximations 
    $\mathbf{v}_h^P$ and $\phi_h^P$ that control the divergence, such as in \cite{boffi2013convergence}. 
    However, the projections $\pLLM[1]$ and $\pLLM[2]$, that we implement and test numerically, 
    do not satisfy Assumption \ref{assumption_3}, therefore we directly move to next section, 
    dedicated to error estimates for these projections.
\end{rmk}

\subsection{Weaker approximation assumptions}
\label{sec:sub_error_analysis_with_relaxed_assumptions}
In this section we proof a convergence result for weaker assumptions than the ones in Assumption \ref{assumption_3}. 
The reason is the following. Of considerable practical interest are projections in discrete spaces that have good 
approximation properties and are computationally fast and efficient to compute, i.e. \cite{kraus2017gempic,holderied2021mhd}. 
A very general family of such projections are the quasi-interpolators proposed in \cite{lee2000some}.
Since we do not have good approximation properties of the adjoint projection of such quasi-interpolants, 
we will require here to relax the assumptions we made in the previous section.
{ \begin{ass}
\label{assumption_4}
The projections $\Pi^1$ and $\Pi^{2}$ satisfy
\begin{equation*}
    \|\Pi^1\|_{\mathbf{L}^{\infty}(\widetilde{K}) \to \mathbf{L}^{\infty}(K)} = C_1 ,\quad \text{and} \quad \|\Pi^2\|_{{L}^{\infty}(\widetilde{K}) \to {L}^{\infty}(K)} = C_2.
\end{equation*}
where $K\in \mathcal{M}$ and $\widetilde{K}$ is its support extension.
Moreover, given $\mathbf{w}\in \mathbf{H}^m(\mathrm{div};\Omega)$ and 
$\psi \in H^m(\Omega)$, with at least $m\geq2$, for $0\leq l < m \leq p $, and 
$p = \min_{i=1}^d\{p_i\}$, there exists a constant $C$ such 
that it holds:
\begin{subequations}\label{eqn:stime_sobolev}
    \begin{align}
        \label{eqn:stima_hdiv}
        \|(Id - \Pi^1)\mathbf{w}\|_{\mathbf{H}^l(\mathrm{div}; \Omega)}  &\leq C h^{m-l} \| \mathbf{w}\|_{\mathbf{H}^m(\mathrm{div};\Omega)},\\ 
        \label{eqn:stima_hm}
        \|(Id - \Pi^2)\psi\|_{{H}^l(\Omega)}  &\leq C h^{m-l} \| \psi \|_{{H}^m(\Omega)},
    \end{align}        
\end{subequations}
where $Id$ is the identity operator, and $h$ is the mesh size. 
\end{ass}}
{ Notice that, the projectors $\pLLM[1]$ and $\pLLM[2]$ satisfy Assumption \ref{assumption_4}, but not Assumption \ref{assumption_3}. 
We consider again $\mathbf{v}_h^P = \Pi^1(\mathbf{v}) $ and $\phi_h^P = \Pi^2(\phi)$, and we
notice that under Assumption \ref{assumption_4} it still holds \cref{lemma5.1}. Therefore we have the following theorem.}

\begin{thm}\label{theroem5.3}
Under Assumption \ref{assumption_4}, together with the commutativity of the projectors, 
$\Pi^2 \textup{div} = \textup{div} \Pi^1$, given $\mathbf{v} \in \mathbf{H}^m(\mathrm{div};\Omega)$ and $\phi \in H^m(\Omega)$, 
with at least $m\geq 2$, and $c \in \mathcal{C}^{\infty}(\Omega)$, it holds that
\begin{equation}
\label{tesi.weak}
    \|\mathbf{v}-\mathbf{v}_h,\phi-\phi_h\|_{\infty,E} \leq C h \|\mathbf{v},\phi\|_{W^{1,\infty},\mathcal{H}^m}.
\end{equation}
\end{thm}
\begin{proof}
The first steps of the proof are as in Theorem \ref{Theorem5.2} until the choice of the test functions, { and in particular \eqref{eq:same_steps} is valid. Then, by construction we have $\Pi^2 \textup{div} = \textup{div}\Pi^1 $, and by integrating by parts we obtain}
\begin{align*}
    \left((Id-\Pi^2)\textup{div}(c\mathbf{v}_h-c\mathbf{v}_h^P),\phi\right) =& -\left((Id-\Pi^1)(c\mathbf{v}_h-c\mathbf{v}_h^P),\grad{\phi}\right)\\
    \left((Id-\Pi^2)\textup{div}(c\mathbf{v}),\phi_h-\phi_h^P\right) =& \left(\textup{div}(Id-\Pi^1)(c\mathbf{v}),\phi_h-\phi_h^P\right),
\end{align*}
note the difference in the first equation with respect to \eqref{eq:diff_step}.
By using Cauchy-Schwartz inequality 
as in the proof of Theorem~\ref{Theorem5.2}, we have 
\begin{align*}
    \begin{split}
        \dfrac{1}{2}\pdv{}{t}\|\mathbf{v}_h - \mathbf{v}_h^P, \phi_h - \phi_h^P\|_{E}^2 \leq \|(\mathbf{v}-\mathbf{v}_h^P)_t\|_{\mathbf{L}^2(\Omega)} \|\mathbf{v}_h-\mathbf{v}_h^P\|_{\mathbf{L}^2(\Omega)} +\\
        +       \|(\phi-\phi_h^P)_t\|_{L^2(\Omega)}    \|\phi_h-\phi_h^P\|_{L^2(\Omega)}+\\ 
        +       \|\phi-\phi_h^P \|_{L^2(\Omega)}  { \|\Pi^2\textup{div}(c\mathbf{v}_h-c\mathbf{v}_h^P)\|_{\mathbf{L}^2(\Omega)}} + \\
        +   { \|\Pi^2\textup{div}(c\mathbf{v} - c\mathbf{v}_h^P)\|_{L^2(\Omega)}} \|\phi_h-\phi_h^P\|_{L^2(\Omega)}+\\
        +       \|\phi\|_{H^1(\Omega)}  \|(Id-\Pi^1)(c\mathbf{v}_h-c\mathbf{v}_h^P)\|_{\mathbf{L}^2(\Omega)} +\\
        +       \|\textup{div}(Id-\Pi^1)(c\mathbf{v})\|_{L^2(\Omega)} \|\phi_h-\phi_h^P\|_{L^2(\Omega)},
    \end{split}
\end{align*}
{and the only differences with respect to the proof of the previous theorem are in the stability assumptions of $\Pi^2$ and in the fifth term of the right-hand side.
In order to bound this last one,} we can apply the superconvergence results stated in \cite[Theorem~2.2]{bertoluzza1999discrete}, that is
\begin{equation}
\label{eqn:5.10}
    \|(Id-\Pi^1)(c\mathbf{v}_h-c\mathbf{v}_h^P)\|_{\mathbf{L}^2(\Omega)} \leq  C h \|\mathbf{v}_h-\mathbf{v}_h^P\|_{\mathbf{L}^2(\Omega)}.
\end{equation}
Next, we bound $\|\Pi^2\textup{div}(c\mathbf{v}_h-c\mathbf{v}_h^P)\|_{\mathbf{L}^2(\Omega)}$ by local arguments. Given $K \in \mathcal{M}$, since $|K|\leq h^d$ where $d$ is the dimension of $\Omega$, we first apply H{\"o}lder inequality, 
then stability of Assumption \ref{assumption_4} together with uniform boundedness of $c$ and finally inverse 
inequalities for splines, that is
\begin{align*}
    \begin{split}
        \|\Pi^2\textup{div}(c\mathbf{v}_h-c\mathbf{v}_h^P)\|_{\mathbf{L}^2(K)}^2 &\leq |K|\, \|\Pi^2\textup{div}(c\mathbf{v}_h-c\mathbf{v}_h^P)\|_{\mathbf{L}^{\infty}(K)}^2 \\
        &\leq h^d C_2^2     \|\nabla c \cdot (\mathbf{v}_h-\mathbf{v}_h^P) + c\, \textup{div}(\mathbf{v}_h-\mathbf{v}_h^P)\|_{{L}^{\infty}(\widetilde{K})}^2\\
        &\leq h^d C \left(  \|\mathbf{v}_h-\mathbf{v}_h^P \|_{\mathbf{L}^{\infty}(\widetilde{K})} + \|\textup{div} (\mathbf{v}_h-\mathbf{v}_h^P) \|_{{L}^{\infty}(\widetilde{K})}\right)^2\\
        &\leq h^d C \left(  \|\mathbf{v}_h-\mathbf{v}_h^P \|_{\mathbf{L}^{\infty}(\widetilde{K})} + h^{-1}\|\mathbf{v}_h-\mathbf{v}_h^P\|_{\mathbf{L}^{\infty}(\widetilde{K})}\right)^2\\
        &\leq h^d C h^{-2}  \|\mathbf{v}_h-\mathbf{v}_h^P \|_{\mathbf{L}^{\infty}(\widetilde{K})}^2\\
        &\leq C h^{-2}\|\mathbf{v}_h-\mathbf{v}_h^P \|_{\mathbf{L}^{2}(\widetilde{K})}^2.
    \end{split}
\end{align*}
{ Notice that in the last inequality it is essential to have a shape regularity assumption of the mesh.}  
Taking the square root 
and by standard arguments we have immediately the global inequality
$$
\|\Pi^2\textup{div}(c\mathbf{v}_h-c\mathbf{v}_h^P)\|_{\mathbf{L}^2(\Omega)} \leq C{ h^{-1}} \|\mathbf{v}_h-\mathbf{v}_h^P \|_{\mathbf{L}^{2}(\Omega)}.
$$
Recall that $\mathbf{v}_h^P = \Pi^1 (\mathbf{v})$ and $\phi_h^P = \Pi^2(\phi)$, 
together with $\textup{div} \Pi^1 = \Pi ^2 \textup{div}$. Again, with the same arguments, 
and with $|K|\leq h^d$, we have 
\begin{align*}
    \begin{split}
        \|\Pi^2\textup{div}(c\mathbf{v} - c\mathbf{v}_h^P) \|_{L^2({K})}^2 &\leq h^d\|\Pi^2\textup{div}(c\mathbf{v} - c\mathbf{v}_h^P)\|_{L^{\infty}(K)}^2 \\
        &\leq h^dC_2^2 \|\nabla c \cdot (\mathbf{v}-\mathbf{v}_h^P) + c\, \textup{div}(\mathbf{v}-\mathbf{v}_h^P)\|_{L^{\infty}(\widetilde{K})}^2 \\
        &\leq h^dC \left(\|\mathbf{v}-\mathbf{v}_h^P \|_{\mathbf{L}^{\infty}(\widetilde{K})} + \|\textup{div} (\mathbf{v}-\mathbf{v}_h^P) \|_{{L}^{\infty}(\widetilde{K})}\right)^2\\
        &\leq h^dC \left(\|(Id- \Pi^1)\mathbf{v} \|_{\mathbf{L}^{\infty}(\widetilde{K})}^2 + \|(Id-\Pi^2)\textup{div}(\mathbf{v}) \|_{{L}^{\infty}(\widetilde{K})}^2\right)\\
        &\leq h^dC \left( h^{2m-d} \|\mathbf{v} \|_{\mathbf{H}^{m}(\widetilde{\widetilde{K}})}^2 +  h^{2m-d} \|\textup{div}(\mathbf{v}) \|_{{H}^{m}(\widetilde{\widetilde{K}})}^2\right)\\
        &\leq C h^{2m} \|\mathbf{v} \|_{H^m(\textup{div};\widetilde{\widetilde{K}})}^2,
    \end{split}
\end{align*}
where { $\widetilde{\widetilde{K}}$ is the support extension of $\widetilde{K}$}. Notice that we used Cauchy-Schwartz inequality on fourth row, 
and the local approximation estimates of the kind \cite[Theorem 10.2]{lyche1975local} on fifth row. Taking the square root and by standard arguments this global inequality is straightforward 
$$
\|\Pi^2\textup{div}(c\mathbf{v} - c\mathbf{v}_h^P)   \|_{L^2({\Omega})} \leq Ch^{m} \|\mathbf{v} \|_{H^m(\textup{div};\Omega)}.
$$
We bound the remaining of approximation errors as it is done for Lemma \ref{lemma5.1}, and together with the above inequalities, we have
\begin{align*}
    \begin{split}
        \dfrac{1}{2}\pdv{}{t}\|\mathbf{v}_h - \mathbf{v}_h^P, \phi_h - \phi_h^P\|_{E}^2 \leq Ch^m\|\mathbf{v}_t\|_{\mathbf{H}^m(\textup{div}; \Omega)} &\|\mathbf{v}_h-\mathbf{v}_h^P\|_{\mathbf{L}^2(\Omega)} +\\
        + Ch^m     \|\phi_t\|_{H^m(\Omega)}  & \|\phi_h-\phi_h^P\|_{L^2(\Omega)}+\\ 
        +Ch^{m-1}  \|\phi\|_{H^m(\Omega)}    & \|\mathbf{v}_h-\mathbf{v}_h^P\|_{\mathbf{L}^2(\Omega)}+\\
        +Ch^{m}    \|\mathbf{v}\|_{\mathbf{H}^m(\textup{div}; \Omega)}  & \|\phi_h-\phi_h^P\|_{L^2(\Omega)}+\\
        +Ch        \|\phi\|_{H^1(\Omega)}    & \|\mathbf{v}_h-\mathbf{v}_h^P\|_{\mathbf{L}^2(\Omega)}+\\
        +Ch^m      \|\mathbf{v}\|_{\mathbf{H}^m(\textup{div}; \Omega)} &\|\phi_h-\phi_h^P\|_{L^2(\Omega)}.
    \end{split}
\end{align*}    
By using Cauchy-Schwartz inequality and adding together the similar terms, we have 
\begin{align*}
        \dfrac{1}{2}\pdv{}{t}\|\mathbf{v}_h - \mathbf{v}_h^P, \phi_h - \phi_h^P\|_{E}^2  
        & \leq Ch^m \|\mathbf{v}_t,\ \phi_t\|_{\mathcal{H}^m} \|\mathbf{v}_h-\mathbf{v}_h^P,\phi_h-\phi_h^P\|_{E}\\ 
        &   +  Ch \|\mathbf{v},\ \phi\|_{\mathcal{H}^m}     \|\mathbf{v}_h-\mathbf{v}_h^P,\phi_h-\phi_h^P\|_{E}.
\end{align*} 
Dividing by $\|\mathbf{v}_h-\mathbf{v}_h^P,\phi_h-\phi_h^P\|_{E}$ and integrating in time, we get 
\begin{align*}
    \|\mathbf{v}_h - \mathbf{v}_h^P, \phi_h - \phi_h^P\|_{E} &\leq \|\mathbf{v}_h(\boldsymbol{x},0) - \mathbf{v}_h^P(\boldsymbol{x},0), \phi_h(\boldsymbol{x},0) - \phi_h^P(\boldsymbol{x},0)\|_{E} + \\ 
    & + Ch^m \int_0^t \|\mathbf{v}_t,\ \phi_t\|_{\mathcal{H}^m} \mathrm{d}\tau + Ch \int_0^t  \|\mathbf{v},\ \phi\|_{\mathcal{H}^m} \mathrm{d}\tau.
\end{align*}
Specifying the initial condition for problem \eqref{eqn:5.1} to be $\left(\mathbf{v}_h^P(\boldsymbol{x},0), \phi_h^P(\boldsymbol{x},0)\right)$, 
it holds $ \|\mathbf{v}_h(\boldsymbol{x},0) - \mathbf{v}_h^P(\boldsymbol{x},0),\ \phi_h(\boldsymbol{x},0) - \phi_h^P(\boldsymbol{x},0)\|_{E}=0$. 
Taking the maximum over $[0,T]$, we have
\begin{equation}\label{eqn:5.12}
    \|\mathbf{v}_h - \mathbf{v}_h^P, \phi_h - \phi_h^P\|_{\infty,E} \leq  Ch \|\mathbf{v},\ \phi\|_{W^{1,\infty},\mathcal{H}^m}.
\end{equation}
By putting together \eqref{eqn:5.3.5}, Lemma \ref{lemma5.1} and \eqref{eqn:5.12} we end up proving \eqref{tesi.weak}.
\end{proof}
{
\begin{rmk}
    We have proved linear convergence under $h$ refinement for the semi-discretization \eqref{eqn:semi_disc_in_spazio}
    using projections as in Assumption \ref{assumption_4}. This is the case of the projections $\pLLM[1]$ and $\pLLM[2]$. 
    However, in Section \ref{sec:numerical_simulations_in_2d} we investigate numerically this error bound, and show high 
    order rates of convergence. There is numerical evidence that 
    $\left( (Id-\Pi^1) (c\mathbf{v}_h - c\mathbf{v_h^P}), \nabla \phi \right) \approx h^{p}$, while 
    $\|(Id-\Pi^1)(c\mathbf{v}_h-c\mathbf{v}_h^P)\|_{\mathbf{L}^2(\Omega)}$ depends linearly on the mesh size $h$.  
\end{rmk}
}

\section{Implementation}
\label{sec:implementation_and_numerical_results}
{ In this section we introduce the matrix form associated to \eqref{eqn:5.8}, giving further details about the computation of the projections. 
Let us start from equation \eqref{eqn:5.8}, that must hold for all $ \mathbf{w}_h \in X_h^1$. 
Consider as test functions $\{\mathbf{b}_{i,h}\}_{i = 1}^{N}$, the basis functions of $X_h^1$, 
which are the push-forward with the Piola transformation map $\iota^1$ of the B-splines on the parametric domain.  
To assemble the matrices involved in \eqref{eqn:5.8} we compute the projections of the basis functions}
\begin{equation*}
    \Pi^1(c\mathbf{b}_{i,h}) = \sum_{l = 1 }^{N} \theta^i_l \mathbf{b}_{l,h},
\end{equation*} 
and we denote by $\boldsymbol{\theta}^i = (\theta^i_1,\dots, \theta^i_{N})^T$, the column vectors with the coefficients of
the projections $\Pi^1(c\mathbf{b}_{i,h})$ into the space $X^1_h$.
{We then define the matrix $\Tilde{\mathbf{A}} \in \mathbb{R}^{N\times N}$ that encapsulates the second term of \eqref{eqn:5.8}, and with the previous notation each entry of the matrix can be computed as
\begin{equation*}
    \begin{split}
         [\Tilde{\mathbf{A}}]_{i,j} & \coloneqq \int_{\Omega}\text{div}\left(\Pi^{1} (c\mathbf{b}_{i,h})\right)\text{div}\left(\Pi^{1} (c\mathbf{b}_{j,h})\right)\mathrm{d}\mathbf{x}\\
         &= \sum_{l = 1}^N \sum_{m = 1}^N \theta^i_l\theta^j_m\int_{\Omega}\text{div}( \mathbf{b}_{l,h})\text{div}(\mathbf{b}_{m,h})\mathrm{d}\mathbf{x} 
         = (\boldsymbol{\theta}^i)^T \mathbf{A} \boldsymbol{\theta}^j,
    \end{split}
\end{equation*}
for $i,j=1,\ldots, N$, with $\mathbf{A} \in \mathbb{R}^{N\times N}$ defined as}
$[\mathbf{A}]_{l,m} \coloneqq \int_{\Omega}\text{div}( \mathbf{b}_{l,h})\text{div}(\mathbf{b}_{m,h})\mathrm{d}\mathbf{x}$.
{It is therefore convenient to store the coefficients of the projectors in the} matrix $\mathbf{\Theta} = [\boldsymbol{\theta}^1|\dots|\boldsymbol{\theta}^N] \in \mathbb{R}^{N\times N}$, from which we obtain $\Tilde{\mathbf{A}} = \mathbf{\Theta} ^\top \mathbf{A \Theta}$.

{For the computation of the projections, we make use of the commutativity property \eqref{eqn:4.6new}, from which we obtain the two equivalent expressions}
\begin{align*}
    \iota^1 ( \Pi^1(c\mathbf{b}_{i,h})) &= \hat{\Pi}^1 (\iota^1 (c \mathbf{b}_{i,h})) = \hat{\Pi}^1 ( \hat{c} \hat{\mathbf{b}}_{i,h}) = \sum_{l=1}^N \hat{\theta^i_l} \hat{\mathbf{b}}_{l,h},\\
    \iota^1 ( \Pi^1(c\mathbf{b}_{i,h})) &= \iota^1(\sum_{l=1}^N {\theta^i_l} {\mathbf{b}}_{l,h}) = \sum_{l=1}^N {\theta^i_l} \iota^1({\mathbf{b}}_{l,h}) = \sum_{l=1}^N {\theta^i_l} \hat{\mathbf{b}}_{l,h},
\end{align*}
{ where $\hat{\theta}^i_l$ are the coefficients of the projection $\hat{\Pi}^1(\hat{c}\hat{\mathbf{b}}_{i,h})$, and $\hat{c} = c \circ \mathcal{F}$. 
Notice that $\hat{\theta}^i_l = \theta^i_l$, and therefore the projections can be computed in the parametric domain, exploiting the tensor-product structure.} 
The details on the computation of these projections are given in \cref{Appendix}. { We finally recall that, 
in order to commute with the divergence operator, the computation of the projection involves \eqref{eqn:3.4}, 
or \eqref{eqn:3.5} for the periodic case, both of them requiring the application of a quadrature formula.}


Finally, by introducing the notation $\underline{\mathbf{v}}^n$ and 
$\underline{\phi}^n$ for the coefficients of the unknown fields 
$\mathbf{v}_h^n $ and $\phi_h^n$ respectively, equation 
\eqref{eqn:5.8} can be written in the following matrix form:
\begin{equation*}
    \left(\mathbf{M} + \frac{k^2}{4} \mathbf{\Theta}^T \mathbf{A} \mathbf{\Theta}\right) \underline{\mathbf{v}}^{n+1} =     \left(\mathbf{M} - \frac{k^2}{4} \mathbf{\Theta}^T \mathbf{A} \mathbf{\Theta}\right) \underline{\mathbf{v}}^{n} - k \mathbf{\Theta}^T {\mathbf{B}} \underline{\phi}^{n}, \quad \text{for }n = 0,\dots,N-1,
\end{equation*}
where the matrix $\mathbf{M} \in \mathbb{R}^{N \times N}$ denotes 
the mass matrix for the space $X^1_h$, 
and $\mathbf{B} \in \mathbb{R}^{N\times M}$ is defined as 
$ [\mathbf{B}]_{i,j} = \int_{\Omega}b_{j,h}\text{div}(\mathbf{b}_{i,h})\mathrm{d}\mathbf{x}$.
\begin{rmk}
Notice that the computation of the involved matrices is independent 
of time, that is, $\mathbf{M}, \mathbf{A}, \mathbf{B}$ and 
$\mathbf{\Theta}$ {can be computed} only once at the beginning of our method. 
If the coefficient $c$ is time dependent, it is necessary to 
recompute only $\mathbf{\Theta}$ at every time step. 
\end{rmk}


\section{Numerical results}
\label{sec:numerical_simulations_in_2d}

In this section, we have numerically investigated the approximation properties of the method 
by conducting academic tests to determine the convergence order under $h$-refinements. Furthermore, 
concerning the employment of the quasi-interpolant in Section \ref{sec:sub_quasi_interpolant_projections} 
within the described numerical method, even if it is tested on a simplified model problem, we have supplemented 
the array of numerical tests conducted by Holderied \textit{et al.} \cite{holderied2021mhd}. 
Indeed, their studies observed the favorable approximation properties of the quasi-interpolant but not of the 
overall global method. Lastly, we have verified the conservation of the total energy for this method.
All the numerical tests have been performed in Matlab, with the isogeometric analysis open source package GeoPDEs \cite{geopdes3.0}. 

\begin{figure}[t]
    \centering
    \subfloat[]
    [Numerical solution for Dirichlet homogeneous boundary conditions. The arrows show the direction of the velocity field $\mathbf{v}_h$ while the color plot shows the pressure map $\phi_h$.\label{fig:5.3a}]
    {\includegraphics[scale = 0.49]{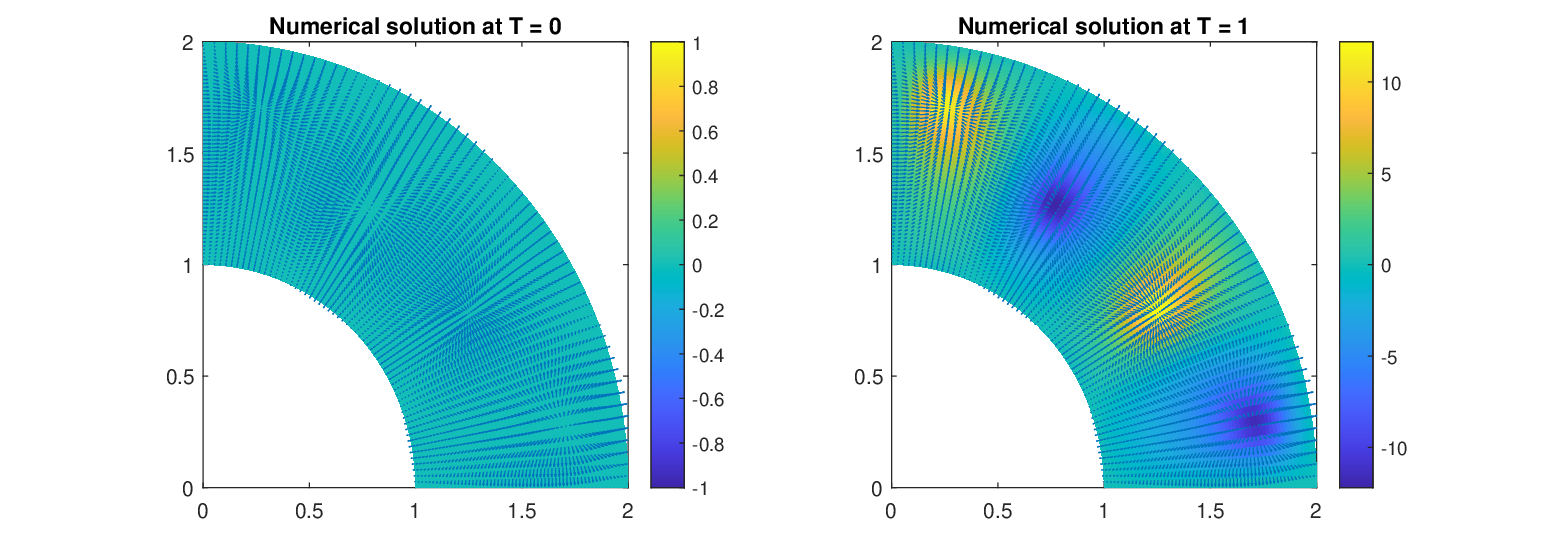}}
    \vskip 0.2cm
    \subfloat[]
    [\textit{Left - } Errors in $\|\cdot\|_{\infty,2}$ norm for solutions with quasi-interpolant and Galerkin methods with homogeneous Dirichlet boundary conditions. \textit{Right - } Errors with $\|\cdot\|_{2,2}$ norm for the same problems.\label{fig:5.3b}]
    {\includegraphics[scale = 0.49]{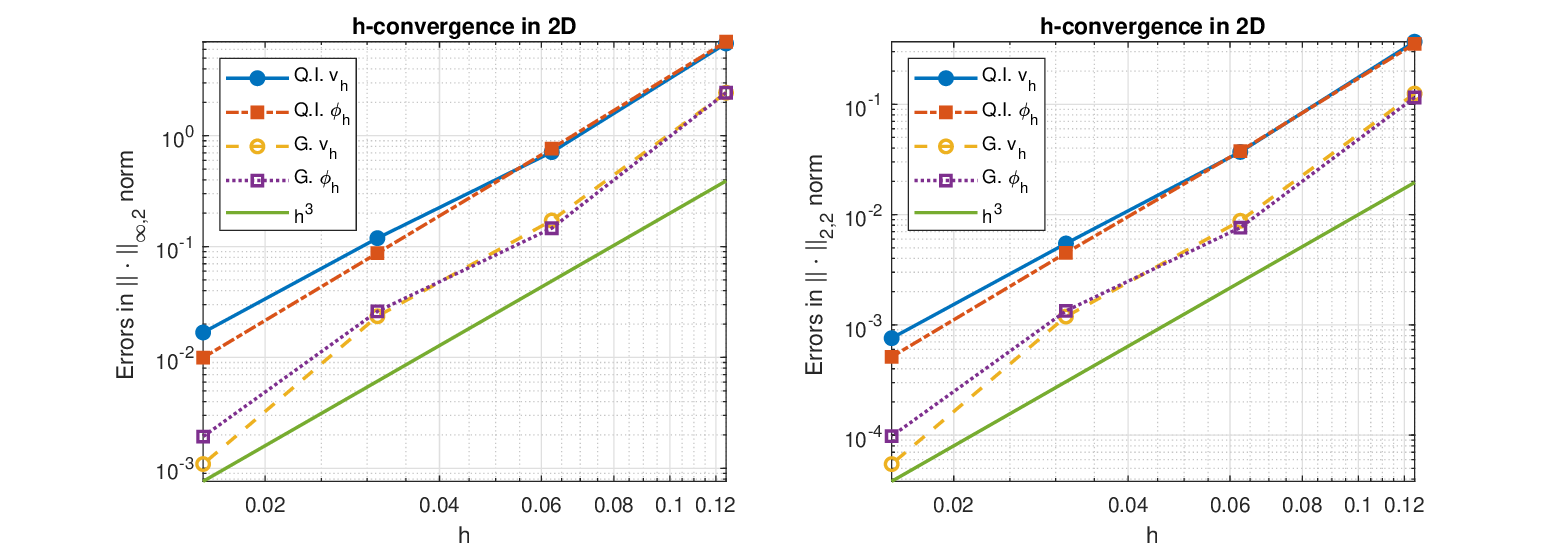}}
    \caption{ (a) Initial data and numerical solutions of \eqref{eqn:5.6} at $T=1$ with ${c} =sin(2\pi {x}_1) sin (2\pi {x}_2) +2$ for Dirichlet homogeneous boundary conditions.  
              (b) $h$-convergence rate estimation for both quasi-interpolant (Q.I.) and Galerkin (G) methods. }
    \label{fig:5.3}
\end{figure}

\subsection{Dirichlet boundary conditions}
\label{sec:sub_Dirichlet_boundary_conditions}
We consider the wave equation as presented in Section \ref{sec:wave}. The domain is $\Omega = \mathcal{F}(\widehat{\Omega})$, where $\widehat{\Omega} = [0,1]^2$ and 
$\mathcal{F}$ is the NURBS map describing the quarter of a ring, as in Figure \ref{fig:new_fig_1}. We assume a full Dirichlet boundary with homogeneous conditions, that is 
$\Gamma_D = \partial\Omega$ and $g = 0$ in \eqref{eqn:2.1}. The time interval is $[0,T]$, where we fix $T=1$. We consider a space dependent coefficient 
$c(x_1,x_2) \coloneqq sin(2\pi {x}_1) sin (2\pi {x}_2) +2 $, which is smooth, with bounded derivatives and bounded away from zero.

We assume that the solution is of the form $u(\boldsymbol{x},t) = \chi(\boldsymbol{x})Re(e^{i\omega t})$, with $\chi \in H^1_0(\Omega)$ solution of 
\begin{equation}\label{eqn:eigenvalue_problem}
    \int_{\Omega} c^2 \nabla \chi \cdot \nabla \Phi \mathrm{d}\boldsymbol{x} = \omega ^2 \int_{\Omega} \chi \Phi \mathrm{d}\boldsymbol{x}, \quad \forall \Phi \in H^{1}_0(\Omega).
\end{equation}
Since $\chi$ is not known explicitly, we use as reference solution its approximation with { splines of degree $\mathbf{p}=(6,6)$} in a uniform mesh with mesh width $h = 1/128$. 
We compute the reference solution, $\chi_h$, relative to the fourth smallest eigenvalue, that is $\omega \approx 4\pi$. Our reference solution is 
$u(\boldsymbol{x},t) = \chi_h(\boldsymbol{x})Re(e^{i\omega t})$, and its velocity and pressure fields are respectively $ \mathbf{v}(\boldsymbol{x},t) = c(\boldsymbol{x})\nabla\chi_h(\boldsymbol{x})Re(e^{i\omega t})$
and $\phi(\boldsymbol{x},t) = \chi_h(\boldsymbol{x})Re(i\omega e^{i\omega t})$.
In order to study the convergence of our method, we estimate the error between the numerical solution and the reference one for different uniform mesh sizes. 
We let the space mesh width $h$ vary in $\{1/8,1/16,1/32,1/64\}$. We choose a uniform partition $\tau $ of the interval $[0,T]$ with $k = 5\times 10^{-4}$. We fix $p=3$, and we set
the initial data from the reference solution, and they can be seen in Figure \ref{fig:5.3a}, for $h = 1/64$. The plot on the right of Figure \ref{fig:5.3a} shows the solutions obtained 
at the final time. For each value of $h$ that we are considering, we record $\|\mathbf{v}-\mathbf{v}_h, \phi-\phi_h\|_{\infty,E}$. 
The plot on the left of Figure \ref{fig:5.3b} shows the two components of the computed errors, that are ${\|\mathbf{v}-\mathbf{v}_h\|_{\infty,2}} \coloneqq \sup_{t \in \tau}\|\mathbf{v}-\mathbf{v}_h\|_{\mathbf{L}^2(\Omega)}  $ 
and ${ \|\phi-\phi_h\|_{\infty,2}} \coloneqq \sup_{t \in \tau}\|\phi-\phi_h\|_{L^2(\Omega)}$. We compare these errors with the ones obtained from solving \eqref{eqn:5.7}, i.e., the standard Galerkin method. 
The plot shows third order convergence rates both for our discretization and for Galerkin.
In order to investigate also convergence with $L^2$ norm in time we will measure the errors in the discrete norm 
$ { \|f(\boldsymbol{x},t)\|_{2,2}} \coloneqq \frac{k}{2}\sum_{t\in \tau \setminus \{T\}}  (\|f(\boldsymbol{x},t)\|_{L^2(\Omega)}^2+\|f(\boldsymbol{x},t+k)\|_{L^2(\Omega)}^2)^\frac{1}{2}$. 
We define the analogous norms for the vector fields. The plot on the right of Figure \ref{fig:5.3b} shows the same error study under 
$h$-refinement using the norm $\|\cdot\|_{2,2}$. We obtain the same error convergence rates, that are of order $h^p$, better then what predicted in the Theorem \ref{theroem5.3}.

\begin{figure}[t]
    \centering
    \subfloat[]
    [Numerical solution for mixed boundary conditions. The arrows show the direction of the velocity field $\mathbf{v}_h$ while the color plot shows the pressure map $\phi_h$.\label{fig:5.4a}]
    {\includegraphics[scale = 0.49]{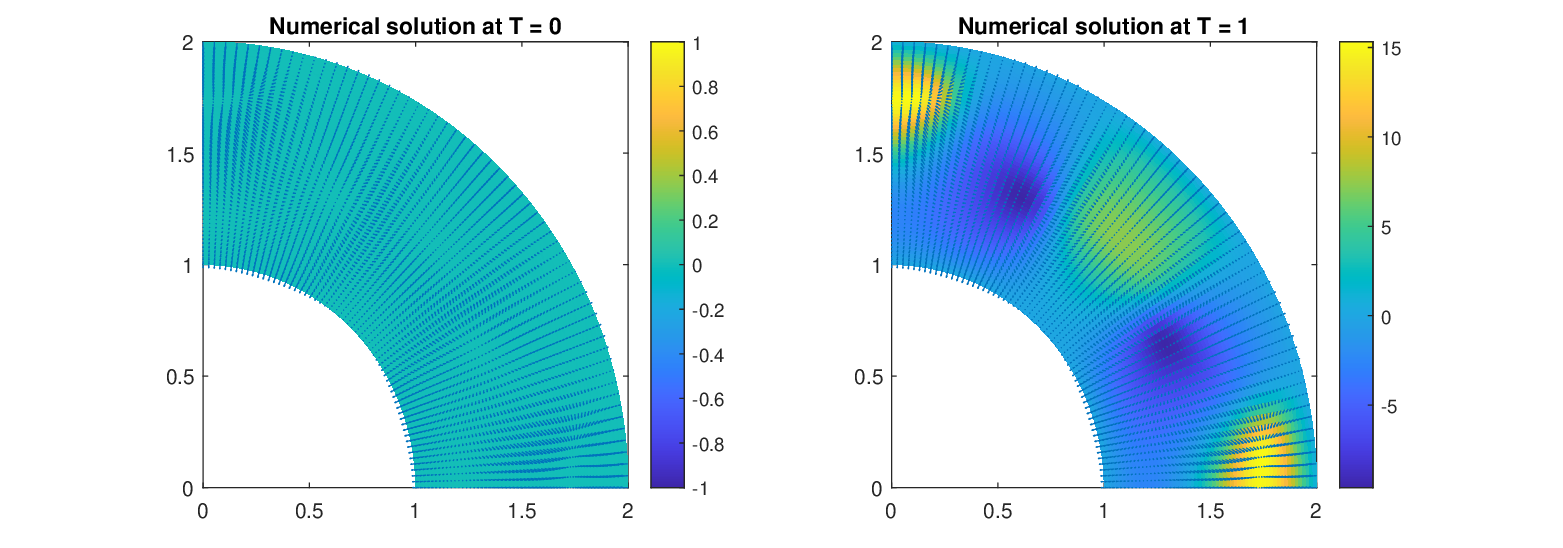}}
    \vskip 0.2cm
    \subfloat[]
    [\textit{Left - } Errors in $\|\cdot\|_{\infty,2}$ norm for solutions with quasi-interpolant and Galerkin methods with mixed boundary conditions. \textit{Right - } Errors with $\|\cdot\|_{2,2}$ norm for the same problems.\label{fig:5.4b}]
    {\includegraphics[scale = 0.49]{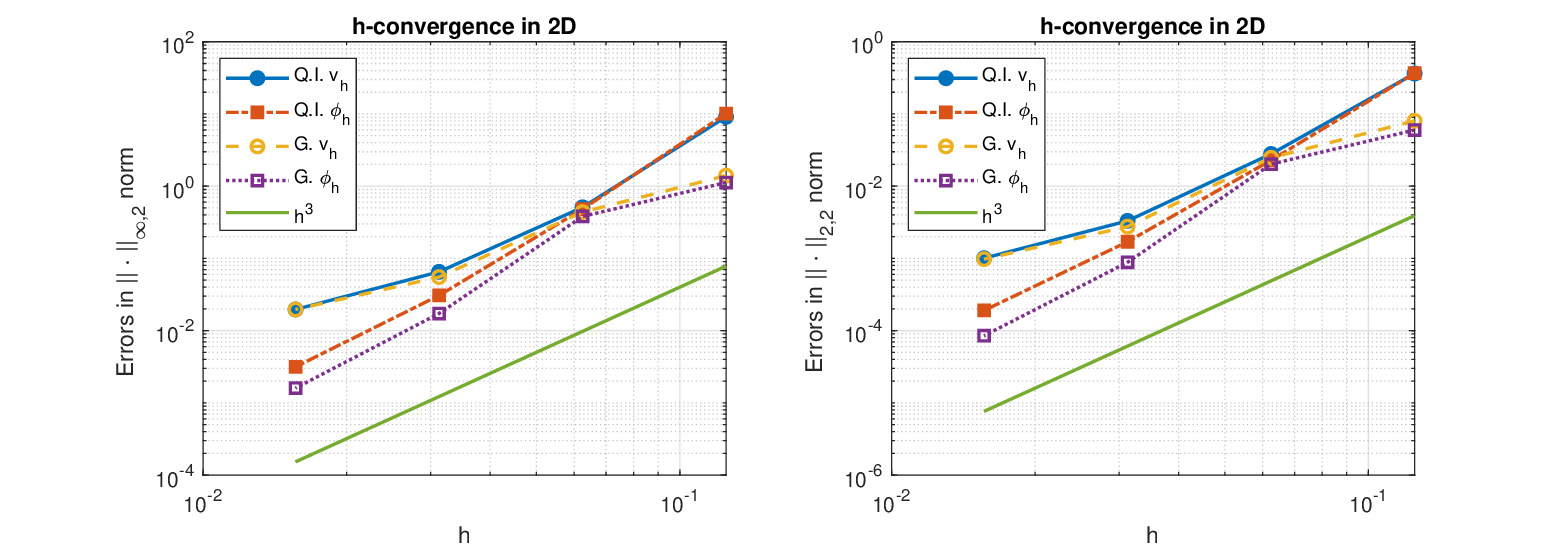}}
    \caption{ (a) Initial data and numerical solutions of \eqref{eqn:5.6} at $T=1$ with ${c} =sin(2\pi {x}_1) sin (2\pi {x}_2) +2$ for mixed boundary conditions.  
              (b) $h$-convergence rate estimation for both quasi-interpolant (Q.I.) and Galerkin (G) methods. }
    \label{fig:5.4}
\end{figure}

\subsection{Mixed boundary conditions}
\label{sec:sub_mixed_boundary_conditions} 
In this second example we consider the domain as above. The boundary of $\Omega$ is defined as in Figure \ref{fig:new_fig_1}, that is $\Gamma_D$ is the Dirichlet boundary, while $\Gamma_P$ is 
the periodic side. The Dirichlet conditions are homogeneous, that is $g =0$. The time interval is $[0,T]$, with $T=1$. Here we used the same coefficient for the example in 
Section \ref{sec:sub_Dirichlet_boundary_conditions}, since it is periodic in $\Gamma_P$. We use separation of variables to construct the reference 
solution, as it was done in the example of Section \ref{sec:sub_Dirichlet_boundary_conditions}. The reference solution is $u(\boldsymbol{x},t)= \chi_h(\boldsymbol{x})Re(e^{i\omega t})$, 
where { $\omega \approx 4\pi$ and $\chi_h $ is the approximation with splines of the solution of \eqref{eqn:eigenvalue_problem} with mixed Dirichlet and periodic boundary condition}.
In order to have a fine approximation, we used $\mathbf{p} = (6,6)$ and a uniform mesh with mesh width $h = 1/128$.  

As it was done for the previous example, we study the convergence of the numerical scheme by estimating the error between the numerical solution and the reference one. 
We discretize the problem as in \eqref{eqn:5.6}, choosing $p=3$ and a uniform mesh. We let the space mesh width $h$ vary in $\{1/8,1/16,1/32,1/64\}$. 
We choose a uniform partition $\tau $ of the interval $[0,T]$ with time step size $k = 5\times 10^{-4}$. 
We compute the initial data from the reference solution, and they can be seen in Figure \ref{fig:5.4a}, for $h = 1/64$. 
The plot on the right of Figure \ref{fig:5.4a} shows the solutions obtained at the final time. 

As it was done for the example in Section~\ref{sec:sub_Dirichlet_boundary_conditions}, for each value of $h$ that we are considering we compute the errors in the energy norm. We also measured the 
errors with the discrete norm { $ \|\cdot\|_{2,2} $}. Figure \ref{fig:5.4b} shows the two components of the computed errors, compared with the ones 
obtained from solving with the standard Galerkin method. Here it seems the error convergence rate is close to a third order of convergence. Although in the last refinement step the convergence is reduced, the same behavior is observed for the solution with the standard Galerkin method.


\begin{figure}[t]
    \centering
    \subfloat[]
    [\textit{Left - } Convergence rates for $\mathbf{v}_h$ and $\phi_h$ measured with $\|\cdot\|_{\infty,2}$ norm. \textit{Right - } Same convergence rates in $\|\cdot\|_{2,2}$ norm. \label{fig:5.5a}]
    {\includegraphics[scale = 0.49]{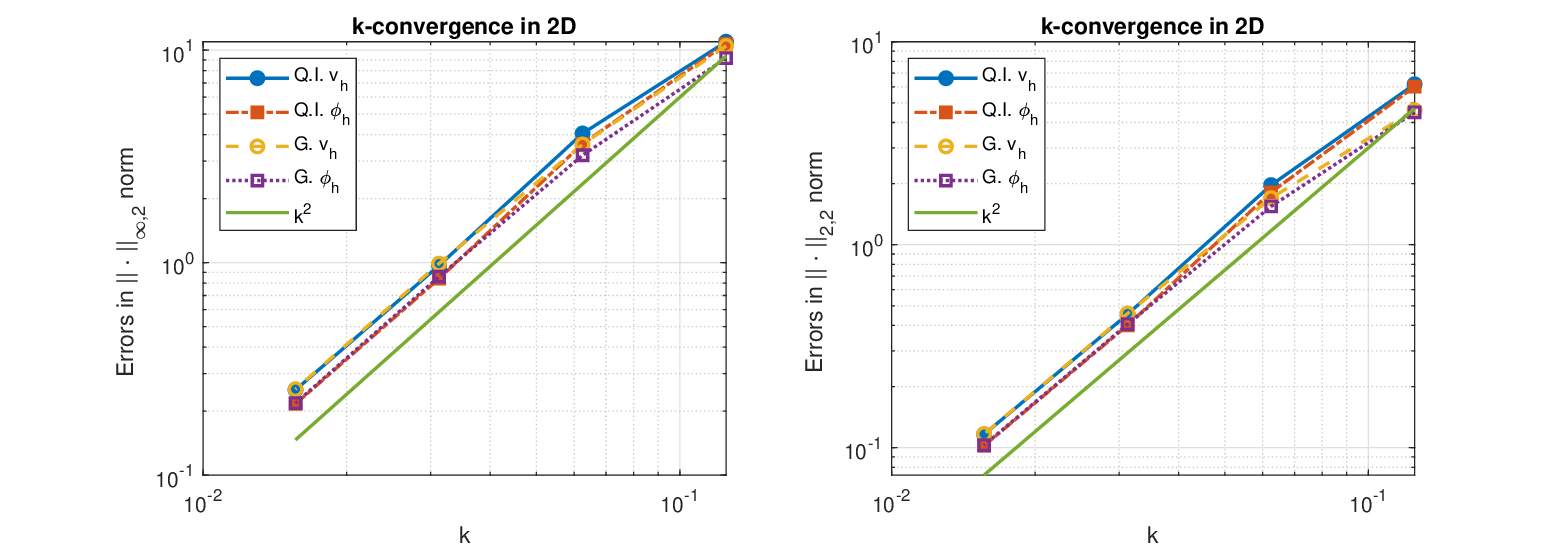}}
    \vskip 0.2cm
    \subfloat[]
    [\textit{Left - } convergence rates for $\mathbf{v}_h$ and $\phi_h$ measured with $\|\cdot\|_{\infty,2}$ norm. \textit{Right - } same convergence rates in $\| \cdot\|_{2,2}$ norm. \label{fig:5.5b}]
    {\includegraphics[scale = 0.49]{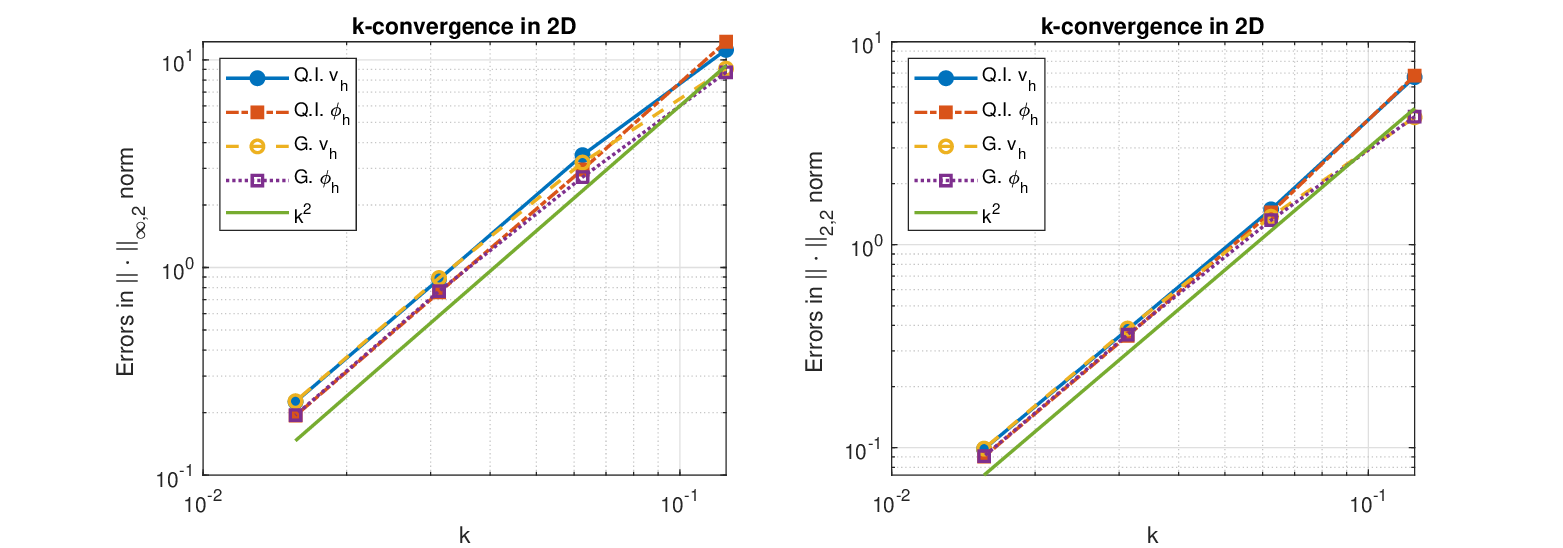}}
    \caption{ Global convergence rate estimation, for quasi-interpolant (Q.I.) and Galerkin (G) methods, for homogeneous Dirichlet boundary conditions (a), while mixed boundary conditions in (b)}
    \label{fig:5.5}
\end{figure}

\subsection{Convergence study under time refinement}
\label{sec:sub_convergence_studi_under_time_refinement}
So long we discussed the $h$-convergence of the quasi-interpolant method. Here we check the global convergence rate when refining both the mesh size $h$ and the time step $k$. 
We let vary $h$ as before from $1/8$ to $1/64$, this time taking the time step $k = h$. We show in Figure \ref{fig:5.5} 
the convergence of both fields $\mathbf{v}_h$ and $\phi_h$ in the same norms introduced above, for the example of Section \ref{sec:sub_Dirichlet_boundary_conditions} in Figure 
\ref{fig:5.5a}, and for the example of Section \ref{sec:sub_mixed_boundary_conditions} in Figure \ref{fig:5.5b}. The obtained convergence rate is of second order, as expected 
from the use of Crank-Nicolson method.


\begin{figure}[t]
    \centering
    \subfloat[]
    [Homogeneous Dirichlet boundary conditions with timesteps $k_1=2e-1$ (in blue) and $k_2 = 1e-2$ (in red). \label{fig:5.6a}]
    {\includegraphics[scale = 0.49]{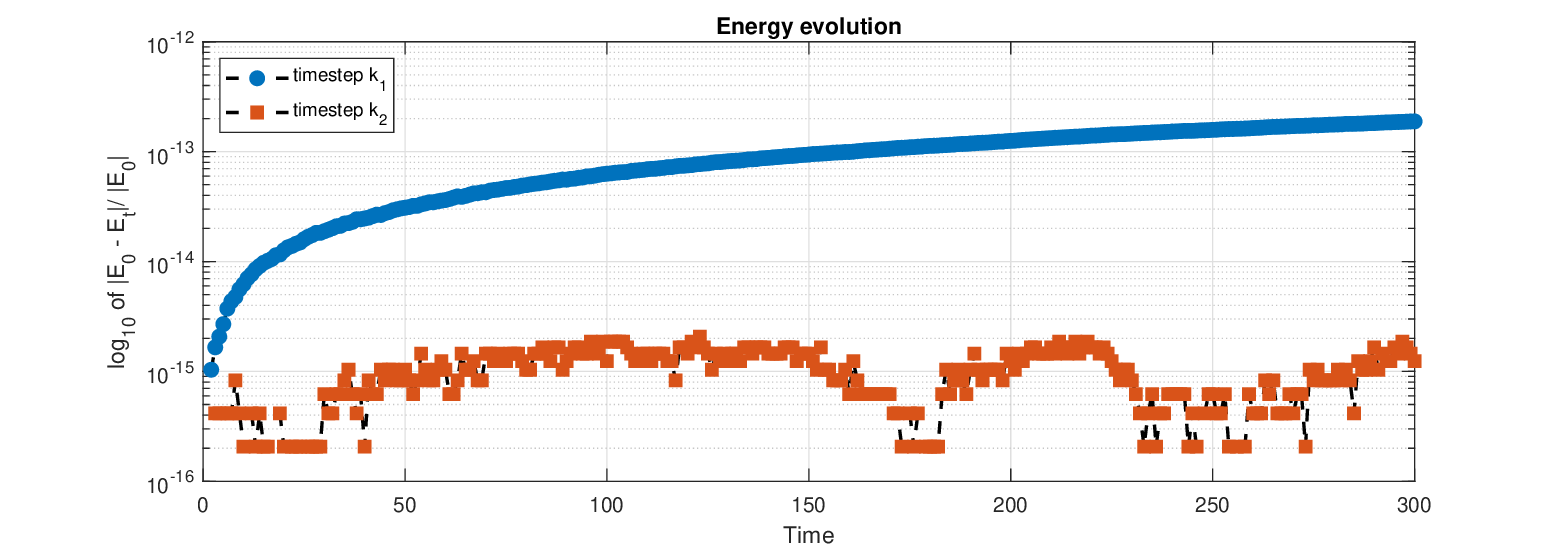}}
    \vskip 0.2cm
    \subfloat[]
    [Mixed boundary conditions with timesteps $k_1=2e-1$ (in blue) and $k_2 = 1e-2$ (in red). \label{fig:5.6b}]
    {\includegraphics[scale = 0.49]{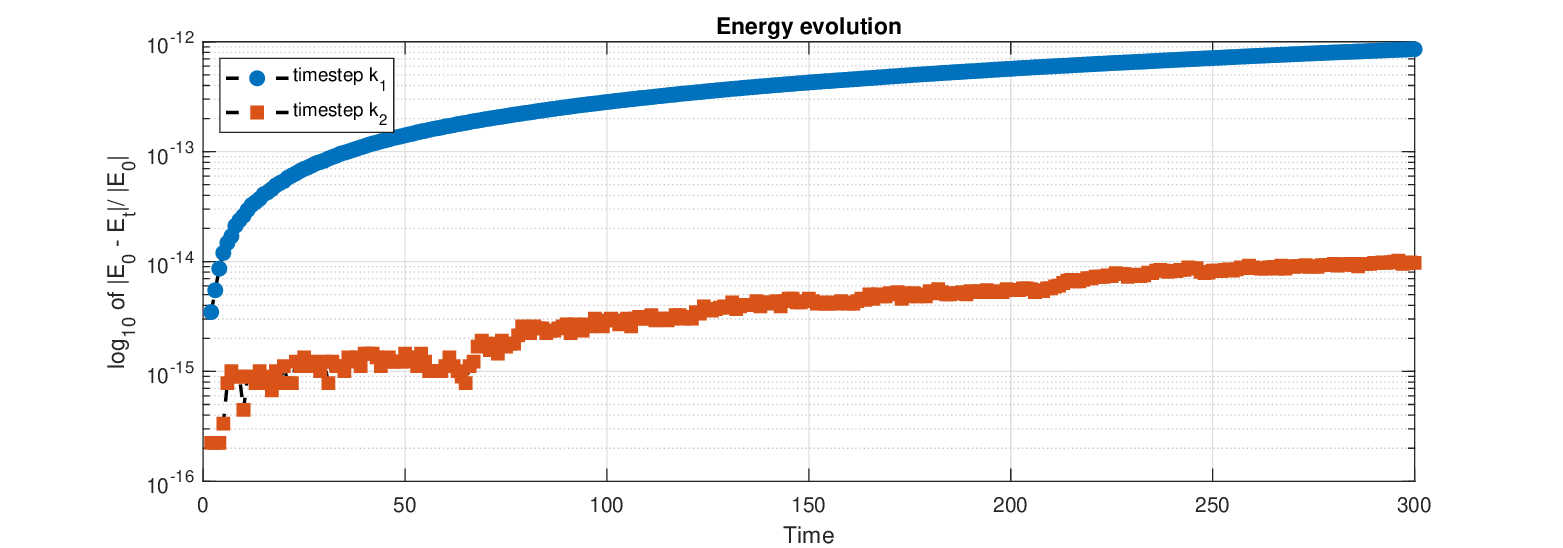}}
    \caption{ Energy conservation plots for Dirichlet homogeneous boundary conditions (a) and mixed boundary conditions (b)}
    \label{fig:5.6}
\end{figure}

\subsection{Energy conservation}
\label{sec:sub_energy_conservation}
In Remark \ref{Remark:energy_conservation} we point out that our numerical scheme is preserving the total energy of the system, as defined in \eqref{eqn:2.5}. 
Here we check the energy conservation for long time simulations, hence we fix $T=300$ and choose two different uniform partitions 
$\tau_1$ and $\tau_2$ of the interval $[0,T]$, the first with step $k_1 = 0.2$, and the second with step size $k_2=0.01$. 
We also fix the meshsize $h =1/32$. Recall that we consider as reference solution a stationary wave with a fixed time-frequency $\omega = 4\pi$, 
and the solution evolves in time as $\Psi(t) \approx Re(e^{i\omega t})$, therefore it performs around two complete oscillations for each unit interval. We solve 
problem \eqref{eqn:5.6} and compute the energy as in \eqref{eqn:2.5} for every $t_n \in \tau \cap \mathbb{Z}$. In Figure \ref{fig:5.6} we show the evolution of the relative 
errors $\frac{|E_0 - E_{t_n}|}{|E_0|}$, where $E_{t_n}$ is the energy evaluated at time $t_n$. Figure \ref{fig:5.6a} shows the evolution of the relative energy errors in 
semi-logarithmic scale for the solution of the problem with homogeneous Dirichlet boundary conditions, as the example in Section 
\ref{sec:sub_Dirichlet_boundary_conditions}. We can see that for both time steps the error remains at the level of round-off errors, 
with lower numbers for the finer time grid. The same plots are reproduced in Figure \ref{fig:5.6b} for the solutions of the problem with mixed boundary conditions of 
Section \refeq{sec:sub_mixed_boundary_conditions}, and we observe a similar behavior. Notice that on the finer grid, for both Dirichlet and mixed boundary conditions, 
we performed $30000$ steps in time without losing energy. The increasing behavior of the errors for the coarse grid solutions seems only due to accumulation of round-off errors. 
We can conclude that the method preserves the total energy of the system as we expected.

\section{Conclusions}
\label{sec:conclusions}

In this paper we proposed an isogeometric discretization of the anisotropic wave equation 
in mixed form, with mixed Dirichlet and periodic boundary conditions.
Our method relies on tensor product projections into spline spaces with good 
approximation properties and that commute with the divergence operator, 
according to the De Rham complex for splines. The conservation of the total energy 
of the system, is imposed weakly by modifying the differential equations according 
to the energy constraint and to projections operators that we introduced.
Regarding time discretization we employed Crank-Nicolson method, which is of second order and 
energy conservative, though alternative conservative methods could be selected.

The method employed for energy preservation finds precedent in the field of magnetohydrodynamics, 
see for example the discretization proposed by\cite{kraus2017gempic} for the Vlasov-Maxwell equations. 
What distinguishes our work, except for the simplified model problem, 
is the a priori error estimate analysis for the full method, not just for the projection approximation. 
In cases where the introduced projections, in addition to commutativity with the divergence operator,
exhibit $L^2$-stability and preserve splines,
we assume good approximation properties of the adjoint projection operator to establish
that the method is of high order - specifically, $h^{m-1}$ assuming the solutions reside in $H^m(\Omega)$ and its vectorial counterpart. 
However, the stability requirement in $L^2$-norm, is not always guaranteed, as exemplified by the projections 
introduced in section \ref{sec:sub_quasi_interpolant_projections}. In light of this, under the less stringent assumption
that the projections are locally stable in the uniform norm, without requiring approximation properties of the adjoint projection operators, 
we prove that the method converges, at least linearly in $h$. 

Numerical tests conducted indicate that, in scenarios where we expect linear convergence, 
the method exhibits high order convergence in $h$, matching the order of convergence achieved 
through a Galerkin approximation without projections. This observation suggests that the proposed 
error estimate might be improved, representing a potential avenue for future development. 

Finally, the conservation of energy is also confirmed by numerical results, where the relative 
error remains of the order of machine precision at the final time.


\appendix
\section{Quasi-interpolant projections} 
\label{Appendix}
\begin{figure}
    \centering
    \subfloat[]
    [Given the function $f_1(x_1,x_2) = (1+x_1)^2sin(\pi x_2/2)$, we highlight the evalutation over knots midpoints and further midpoints respectively in blue green and red. \label{fig:5.1.A}]
    {\includegraphics[scale = 0.38]{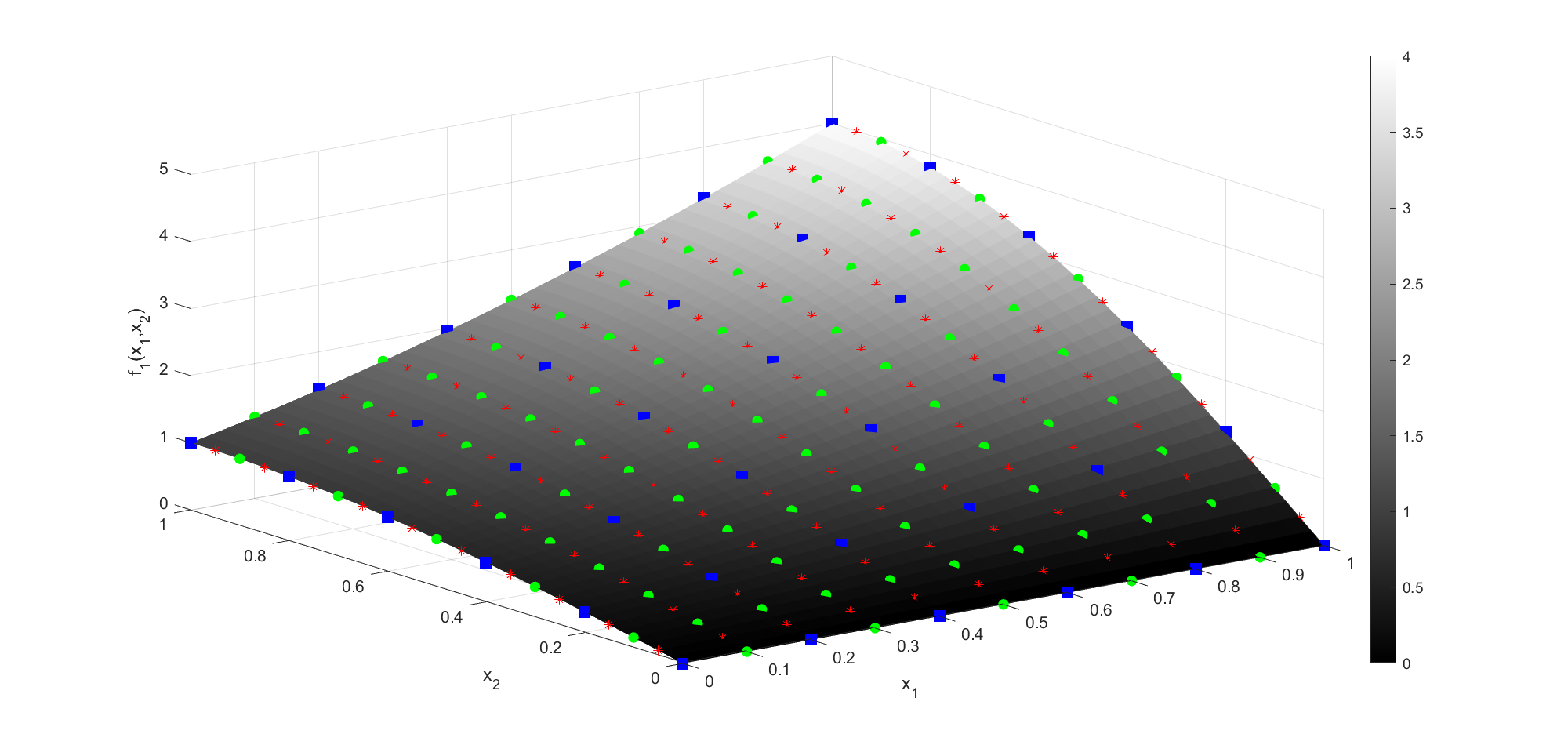}}
    \vskip 0.2cm
    \subfloat[]
    [Here we see the coefficients $\mu_{i_2}(x_1)_{i_2 = 1 }^{n_2-1}$ in blue and green for knots and midpoint. \label{fig:5.1.B}]
    {\includegraphics[scale = 0.38]{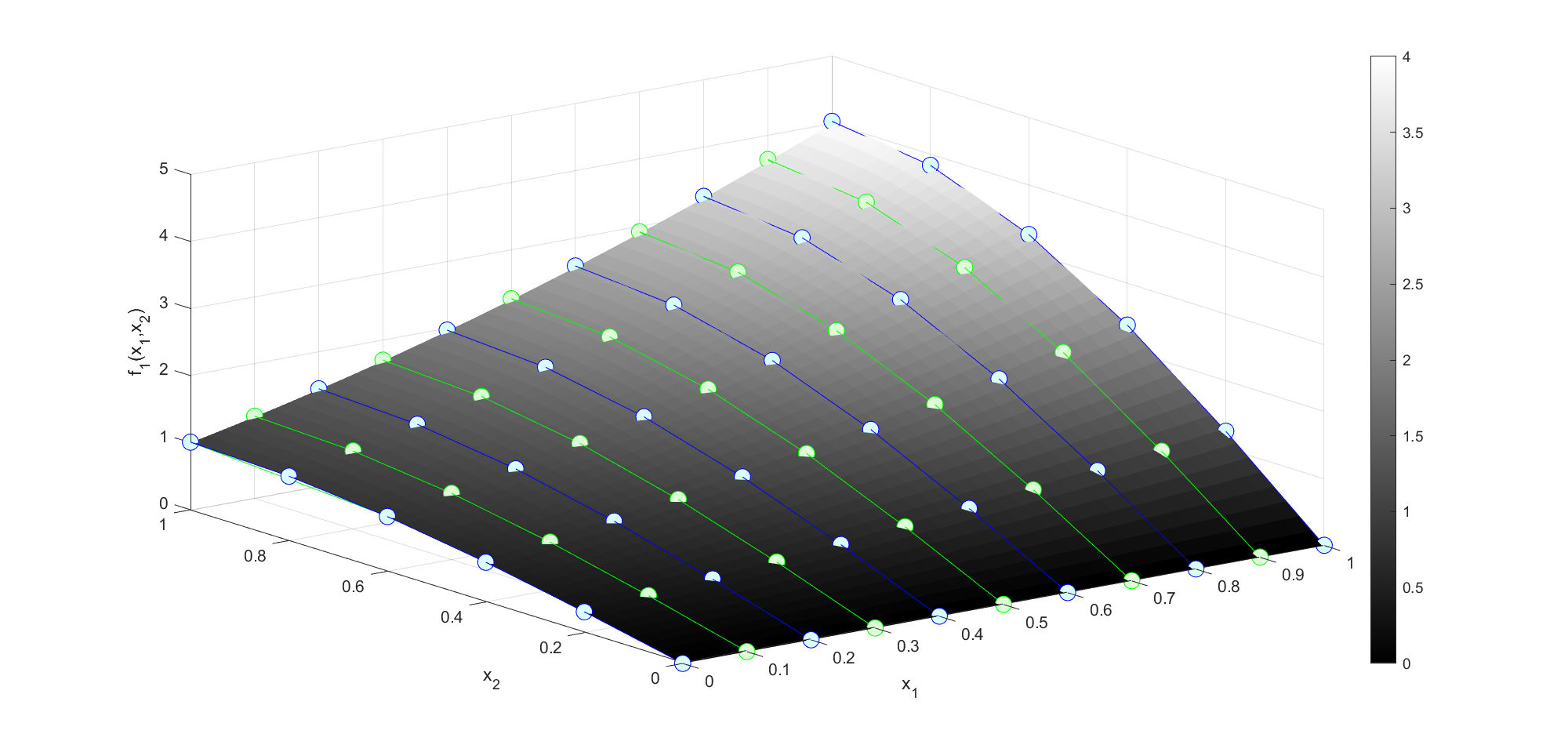}}
    \caption{Graphical visualization of intermediate steps in the computation of $\mu_{i_1,i_2}$. }
    \label{fig:5.1.ABC}
\end{figure}

\subsection{Computation of quasi-interpolants}
\label{sec:sub_computation_of_quasi_interpolants}
In this section we want to discuss the implementation of $\hat{\Pi}^1$, as given in \eqref{eqn:4.3}. The 
construction of the projection in \eqref{eqn:4.13} is analogous. Let us recall that 
$$
\hat{\Pi}^1 = \pi_{p_1} \otimes \pi_{p_2-1}^c \times \pi_{p_1-1}^c \otimes \pi_{p_2}.  
$$
This is applied to project bivariate vector functions $\mathbf{f}= \left(f_1,f_2\right)$, that in our practical 
case are the functions $\hat{c}\mathbf{\hat{b}}_{i,h}$ as in Section \ref{sec:implementation_and_numerical_results}. 
This means that we need to compute $\pi_{p_1}\otimes\pi^c_{p_2-1}(f_1)$ and $\pi^c_{p_1-1}\otimes\pi_{p_2}(f_2)$. 
In order to alleviate notation we fix $p = p_1 = p_2$, and omit the subscript $p$ from the univariate projections. 
With some abuse of notation, these tensorizations must be interpreted as commuting compositions of the kind
\begin{equation*}
    \pi \otimes \pi^c(f_1) = \pi \big(\pi^c (f_1)\big),
\end{equation*}
Notice that $\pi^c$ projection is applied to $f_1$ seen as a function of $x_2$ for a fixed $x_1$, 
while the $\pi$ projection is applied at the result of the previous computation by seeing it as a function of $x_1$ 
for any fixed $x_2$, as it is specified in \cite[Section 2.2.2]{da2014mathematical}. 
We focus on applying $\pi \otimes \pi^c (f_1)$, that is:
\begin{align*}
    \pi\Big(\pi^c\big(f_1(x_1,x_2)\big)\Big)
    &= \pi\left(\sum_{i_2=1}^{n_2-1}\mu_{i_2}(x_1)\hat{D}_{i_2}(x_2)\right)\\ &= \sum_{i_2=1}^{n_2-1}\pi\big(\mu_{i_2}(x_1)\big)\hat{D}_{i_2}(x_2)\\ &= \sum_{i_2=1}^{n_2-1}\sum_{i_1=1}^{n_1} \mu_{i_1,i_2}\hat{B}_{i_1}(x_1)\hat{D}_{i_2}(x_2),
\end{align*}
and the coefficients $\mu_{i_1,i_2}$ will depend only on $f_1$. In order to compute these coefficients 
$\mu_{i_1,i_2}$ we may perform the following steps:
\begin{enumerate}
    \item First evaluate $f_1$ on the Cartesian grid given by breakpoints and midpoints of the first univariate 
            direction, and breakpoints, midpoints and further midpoints of the second univariate direction, 
            see Figure \ref{fig:5.1.A}.
    \item Use these pointwise evaluations to project $f_1(x_1,\cdot)$ with $\pi^c$. This means to perform 
            $\pi^c(f_1(x_1, \cdot ))$ for every $x_1$ in the set of breakpoints and midpoint of the first 
            univariate direction. We recall that this requires to use Cavalieri-Simpson composite quadrature 
            formula. We end up with a set of coefficients $\{\mu_{i_2}(x_1)\}_{i_2=1}^{n_2-1}$, for every 
            fixed $x_1$. In Figure \ref{fig:5.1.B}, these coefficients are highlighted in blue and green 
            circles, respectively for breakpoints and midpoints of $x_1$ direction, for the case $p=2$.
    \item Due to the choice of step 1, for a fixed index $i_2$, the function $\mu_{i_2}(x_1)$ is already evaluated 
            over the breakpoints and midpoints of $x_1$ direction. Therefore we can perform a projection 
            $\pi(\mu_{i_2})$ for every $i_2 = 1, \dots , n_2-1$. We end up with a set of coefficients 
            $\{\mu_{i_1,i_2}\}_{i_1=1,i_2=1}^{n_1,n_2-1}$ over the parametric domain that uniquely identify the 
            bivariate spline that approximates $f_1$. 
\end{enumerate}
The procedure to compute $\pi^c \otimes \pi(f_2)$ is analogous, we just need to swap the evaluation points required 
per univariate direction. After projecting the two scalar components $f_1$ and $f_2$, we have two sets of 
coefficients whose joint corresponds to the degrees of freedom of the spline function that approximates 
$\mathbf{f}$ in the discrete space $\hat{X}^1_h$.

\section*{Acknowledgments} 

\bibliographystyle{siamplain}
\bibliography{references}
\end{document}